\documentclass[11pt,a4paper]{amsart}

\usepackage{amsthm}

\allowdisplaybreaks

\usepackage[utf8]{inputenc}
\usepackage{bbm,amssymb,mathtools,mathrsfs}

\usepackage[usenames,dvipsnames]{xcolor}

\usepackage[colorlinks,linkcolor=BrickRed,citecolor=OliveGreen,urlcolor=black,hypertexnames=true]{hyperref}

\usepackage[margin=3.5cm]{geometry}

\makeatletter
\renewcommand{\tocsubsection}[3]{%
  \indentlabel{\@ifnotempty{#2}{\hspace*{2.3em}\makebox[2.8em][l]{%
    \ignorespaces#1 #2.\hfill}}}#3}
\makeatother

\DeclareMathOperator{\ad}{ad}
\DeclareMathOperator{\id}{id}
\DeclareMathOperator{\Sym}{Sym}
\DeclareMathOperator{\Tr}{Tr}
\DeclareMathOperator{\Trd}{Trd}
\DeclareMathOperator{\End}{End}
\DeclareMathOperator{\Ann}{Ann}

\DeclareMathOperator{\Supp}{Supp}
\DeclareMathOperator{\sign}{sign}
\DeclareMathOperator{\Nil}{Nil}
\DeclareMathOperator{\Int}{Int}
\DeclareMathOperator{\rk}{rank}
\DeclareMathOperator{\Skew}{Skew}
\DeclareMathOperator{\PD}{PD}

\newcommand{\N}{\mathbb{N}}
\newcommand{\Z}{\mathbb{Z}}
\newcommand{\R}{\mathbb{R}}
\newcommand{\CS}{\mathcal{S}}
\newcommand{\fm}{\mathfrak{m}}
\newcommand{\fp}{\mathfrak{p}}
\newcommand{\io}[1]{\prescript{\iota}{}{#1}}
\newcommand{\iss}[1]{\prescript{\sigma}{}{#1}}
\newcommand{\iop}{\io\!\!\cdot}
\newcommand{\ts}{\mathrm{ts}}
\newcommand{\tsa}{\cdot_\ts}

\newcommand{\CM}{\mathscr{M}}
\newcommand{\CCS}{\mathscr{S}}
\newcommand{\Sper}{\operatorname{Sper}}
\newcommand{\Sperm}{\operatorname{Sper}^{\mathrm{max}}}
\newcommand{\Spec}{\operatorname{Spec}}
\newcommand{\s}{\sigma}
\newcommand{\Sign}{\operatorname{Sign}}
\newcommand{\ox}{\otimes}
\newcommand{\x}{\times}
\newcommand{\Qf}{\mathrm{qf}}
\newcommand{\ve}{\varepsilon}
\newcommand{\vt}{\vartheta}
\newcommand{\qf}[1]{\langle #1\rangle}
\newcommand{\Pf}[1]{\langle\!\langle #1\rangle\!\rangle}
\newcommand{\knp}{\bullet}
\newcommand{\sw}{\mathrm{sw}}

\newcommand{\bH}{\mathfrak{H}}
\newcommand{\op}{\mathrm{op}}
\newcommand{\Id}{\mathrm{id}}

\newcommand{\Herm}{\mathfrak{Herm}}

\newcommand{\vf}{\varphi}

\DeclareMathOperator*{\bigperp}{\raisebox{-.8ex}{\scalebox{2}{$\perp$}}}

\newcommand{\mab}{{\CM_{\bar \alpha}}}

\theoremstyle{plain}

  \newtheorem{thm}{Theorem}[section]
  \newtheorem{lem}[thm]{Lemma}
  \newtheorem{lemma}[thm]{Lemma}

  \newtheorem{cor}[thm]{Corollary}
  \newtheorem{prop}[thm]{Proposition}

\theoremstyle{definition}
  \newtheorem{defi}[thm]{Definition}
  \newtheorem{rem}[thm]{Remark}

\numberwithin{equation}{section}

\begin{document}

\title{Pfister's local-global principle for Azumaya algebras
with involution}

\author[V. Astier]{Vincent Astier}
\author[T. Unger]{Thomas Unger}
\address{School of Mathematics and Statistics, 
University College Dublin, Belfield, Dublin~4, Ireland}
\email{vincent.astier@ucd.ie}
\email{thomas.unger@ucd.ie}

%\date{\today}

\begin{abstract}
  We prove Pfister's local-global principle for hermitian forms over Azumaya
  algebras with involution over semilocal rings, and show in particular that
  the Witt group of nonsingular hermitian forms is $2$-primary torsion. Our
  proof relies on a hermitian version of Sylvester's law of inertia, which is
  obtained from an investigation of the connections between a pairing of
  hermitian forms extensively studied by Garrel, signatures of hermitian forms,
  and positive semidefinite quadratic forms.
\end{abstract}

\maketitle

\tableofcontents

\section{Introduction}

Pfister's local-global principle is a fundamental result in the algebraic
theory of quadratic forms over fields. It states that the torsion in the Witt
ring is $2$-primary, and that a nonsingular
quadratic form represents a torsion element in this ring
if and only if its signature (the difference between the number of positive
and the number of negative entries in any diagonalization of the form
according to Sylvester's ``law of inertia'')
is zero at all orderings of the field.  
The above facts are of course well-known, and can easily be found in
standard references such as \cite{sch} or \cite{LQF2}.

The main result of this paper (Theorem~\ref{PLG-Az})
is 
Pfister's local-global principle for nonsingular
hermitian forms over Azumaya algebras with involution over semilocal
commutative rings in which $2$ is a unit. 
(The special case of central simple algebras with involution was treated in
\cite{L-U} and \cite{Sch-1970}; see also \cite{B-U}.)
The  
assumption that the base ring is semilocal 
is minimal
in the sense that Pfister's local-global principle is known to hold for 
nonsingular quadratic forms over such rings, but not in general. We refer
to \cite{baeza} for more details.

Our version of Sylvester's law of inertia (Theorem~\ref{Sylvester-no-pc})
is used in the proof of the main
result, which is inspired by Marshall's proof of Pfister's local-global
principle in the context of abstract Witt rings, cf. \cite{marshall80}. 
A crucial ingredient in our proof is
a certain pairing of forms investigated by Garrel in his 2023 paper
\cite{garrel-2023}.
We were also very fortunate that we could put many results from
the recent papers  by First 
\cite{first23} and Bayer-Fluckiger, First and Parimala \cite{BFP}
to good use.

\section{Preliminaries}

In this paper all rings are assumed unital and associative with $2$ invertible.
We identify quadratic forms over commutative rings with symmetric bilinear
forms, and assume that all fields are of characteristic different from~$2$.
Rings are not assumed to be commutative unless explicitly indicated. 
Our main references for rings with involution and hermitian forms are
\cite{knus91, first23}.

\subsection{Hermitian forms over rings with involution}
\label{hfri}

Let $(A,\s)$ be a ring with involution and let $\ve \in Z(A)$ be such that
$\ve\s(\ve)=1$. We denote the category of $\ve$-hermitian modules over $(A,\s)$
by $\Herm^\ve(A,\s)$. The objects of $\Herm^\ve(A,\s)$ are pairs $(M,h)$, where
$M$ is a finitely generated projective right $A$-module and $h\colon M\x M \to A$ is
an $\ve$-hermitian form. Since we always assume that $2\in A^\x$, all hermitian
modules are even. We denote the category of nonsingular $\ve$-hermitian
modules over $(A,\s)$ (also known as $\ve$-hermitian spaces) by
$\bH^\ve(A,\s)$. The morphisms of $\Herm^\ve(A,\s)$ and $\bH^\ve(A,\s)$ are the
isometries. We denote isometry by $\simeq$. See 
\cite[I, Sections~2 and 3]{knus91} for
more details. If $\ve=1$, we simply write $\Herm(A,\s)$ and $\bH(A,\s)$.
It is common to say hermitian form instead of hermitian module.

We denote the Witt group of nonsingular $\ve$-hermitian forms over $(A,\s)$ by
$W^\ve(A,\s)$ and note that since $2\in A^\x$, metabolic forms are hyperbolic, cf. 
\cite[Section~2.2]{first23} for a succinct presentation.

For $a\in A^\x$, we denote the inner automorphism 
$A\to A,\ x\mapsto axa^{-1}$ by
$\Int(a)$. 
We define $\Sym^\ve(A,\s):=\{x \in A \mid \s(x)=\ve x\}$ and
$\Sym^\ve (A^\x,\s):=\Sym^\ve(A,\s)\cap A^\x$. We also write $\Sym(A,\s)$
instead of $\Sym^1(A,\s)$ and $\Skew(A,\s)$ instead of $\Sym^{-1}(A,\s)$.
For $a_1,\ldots, a_\ell
\in \Sym^\ve(A,\s)$ we denote by $\qf{a_1,\ldots, a_\ell}_\s$
the diagonal  $\ve$-hermitian form 
\[
  A^\ell\x A^\ell \to A,\ (x,y)\mapsto \sum_{i=1}^\ell
  \s(x_i)a_i y_i.
\]

Let $(M, h) \in \Herm^\ve(A,\s)$. We denote by $D_{(A,\s)}(h):=\{h(x,x) \mid x \in M\}$ the set of elements of $A$
represented by $h$.

If $(M,h)\in \bH^\ve(A,\s)$, the \emph{adjoint involution of $h$} is the
involution $\ad_h$ on the ring $\End_A(M)$ implicitly defined by
\begin{equation}\label{ad}
  h(x, \ad_h(f)(y))= h(f(x), y)
\end{equation}
for all $x,y\in M$ and all $f\in \End_A(M)$, cf. \cite[Section~2.4]{first23}. 
\medskip

Consider a second ring with involution $(B,\tau)$
and let $(S,\iota)$  be a commutative ring with involution such that
$(A,\s)$ and $(B,\tau)$ are $(S,\iota)$-algebras with involution in the sense 
of \cite[I, (1.1)]{knus91}, i.e., $A$ and $B$ are
both $S$-algebras and the involutions $\s$ and $\tau$  are compatible
with $\iota$:
\[\s(sa) = \iota(s) \s(a), \ \tau(sb) = \iota(s) \tau(b), \quad \forall a \in A, b \in B, s \in S.\]
Then $(A \ox_S B, \s \ox \tau)$ is an $(S, \iota) $-algebra with involution and,
if $(M_1,h_1)\in \Herm^{\ve_1}(A,\s)$ and $(M_2,h_2)\in \Herm^{\ve_2}(B,\tau)$,
then
\[
  (M_1\ox_S M_2, h_1\ox h_2) \in \Herm^{\ve_1\ve_2}(A\ox_S B,\s\ox \tau).
\]
If $(M_1, h_1)\simeq (M_1', h_1')$ in $\Herm^{\ve_1}(A,\s)$
and $(M_2, h_2)\simeq (M_2', h_2')$ in $\Herm^{\ve_2}(B,\tau)$,
then
\[
  (M_1\ox_S M_2, h_1\ox h_2)\simeq (M_1'\ox_S M_2', h_1'\ox h_2').
\]

Let $(M,h)\in \Herm^{\ve}(A,\s)$ and $(N,\vf)\in \Herm^{\mu}(S,\iota)$. 
Since $(A\ox_S S, \s\ox\iota)\cong (A,\s)\cong (S\ox_S A, \iota\ox\s)$
as rings with involution, it
is not difficult to see that upon identifying $A\ox_S S$ with $A$, 
\[
  h\ox_S \vf (m_1\ox n_1, m_2\ox n_2)= h(m_1,m_2)\vf(n_1,n_2)
\] 
for all $m_1,m_2\in M$ and $n_1,n_2\in N$,
and
\[
  (M\ox_S N, h\ox_S \vf) \simeq (N\ox_S M, \vf \ox_S h)
\]
in $\Herm^{\ve\mu}(A,\s)$.

\begin{lemma}\label{R-Z-isom}
  Let $R$ be a commutative ring, assume that $A$ is an $R$-algebra and that
  $\s$ is an $R$-linear involution on $A$. Let $\iota=\s|_{Z(A)}$.
  Let $u_1, \ldots, u_k \in R$ and $(M,h)\in \Herm^{\ve}(A,\s)$. Then
    \[\qf{u_1,\ldots,u_k}_\iota \ox_{Z(A)} h \simeq \qf{u_1,\ldots,u_k} 
    \ox_R h\]
  (under the canonical identifications  $Z(A)\ox_{Z(A)}A\cong
  R\ox_R A\cong A$).
\end{lemma}

\begin{proof}
  It suffices to show that $\qf{u}_\iota \ox_{Z(A)} h\simeq \qf{u}\ox_R h$
  for $u\in R$. This follows from the isometries 
  $\qf{u}_\iota \simeq \qf{u} \ox_R \qf{1}_\iota$, 
  $\qf{1}_\iota \ox_{Z(A)} h \simeq h$
  (which are straightforward, using the observations preceding the lemma)
  and associativity of the tensor product.  
\end{proof}

We finish this section with a well-known result for which we could not find
a reference:

\begin{lem}\label{adj-prod}
  Let $R$ be a commutative ring, assume that $A$ and $B$ are $R$-algebras 
  and that
  $\s$ and $\tau$ are $R$-linear involutions on $A$ and $B$, respectively.
  Let $(M,\vf)\in \bH^{\ve_1}(A,\s)$ and $(N,\psi) \in \bH^{\ve_2}(B,\tau)$. 
  Then the map
  \[
    \xi\colon  \End_A(M) \ox_R \End_B(N) \to \End_{A\ox_R B}(M\ox_R N)
  \]
  induced by $\xi(f\ox g)=[x\ox y \mapsto f(x)\ox g(y)]$ yields an isomorphism
  of $R$-algebras with involution 
  \[
    (\End_A(M) \ox_R \End_B(N), \ad_\vf\ox \ad_\psi) \cong
    (\End_{A\ox_R B}(M\ox_R N), \ad_{\vf\ox\psi}).
  \]
\end{lem}

\begin{proof}
  The map $\xi$ is an isomorphism of $R$-algebras by 
  \cite[Theorem~1.3.26 and Corollary~1.3.27]{ford17}. To finish the proof
  it suffices to check that $\ad_{\vf\ox\psi}(\xi(f\ox g))=\xi(\ad_\vf(f)\ox 
  \ad_\psi(g))$ using the definition of adjoint involution \eqref{ad}, which
  is a straightforward computation.
\end{proof}

\subsection{Quadratic \'etale  algebras}\label{sec:quadalg}

Let $R$ be a commutative ring and 
let $S$ be a quadratic \'etale $R$-algebra. 
We recall some results from \cite[Section~1.3]{first23} and 
\cite[I, (1.3.6)]{knus91}.
The algebra $S$ has a unique standard involution $\vt$,
and the trace $\Tr_{S/R}$ satisfies 
\begin{equation}\label{eq:tr}
  \Tr_{S/R}(x)= \vt(x)+x\quad \text{for all $x\in S$.}
\end{equation}

Furthermore, $\Tr_{S/R}$ is an involution trace for $\vt$ (cf.
\cite[I, Proposition~7.3.6]{knus91}) and thus if $h \in \Herm(S,\vt)$ is
nonsingular, then $\Tr_{S/R}(h)$ is nonsingular by 
\cite[I, Proposition~7.2.4]{knus91}. Furthermore, if $h$ is 
hyperbolic, then so is the quadratic form $\Tr_{S/R}(h)$ by the first
paragraph of \cite[p.~41]{knus91}.
The converse holds if $R$ is semilocal by \cite[Corollary~8.3]{first23}.

If $R$ is connected and $S$ is not connected, then $S\cong R\x R$ as
$R$-algebras, and $\vt\colon  (x,y)\mapsto (y,x)$ is the exchange involution.

If $R$ is semilocal, then there exists $\lambda\in S$ such that $\lambda^2 \in
R^\x$, $\vt(\lambda)=-\lambda$, and $\{1,\lambda\}$ is an $R$-basis of 
$S$, cf. \cite[Lemma~1.19]{first23}.

\subsection{Azumaya algebras with involution}
Let $R$ be a commutative ring.
Recall from \cite[III, (5.1)]{knus91}
that an $R$-algebra $A$ is an 
\emph{Azumaya $R$-algebra} if $A$ is a faithful finitely generated projective $R$-module and the 
 \emph{sandwich map} 
\begin{equation}\label{swm}
\sw\colon  A\ox_R  A^{\op}\to \End_R(A),\ a\ox b^\op\mapsto [x\mapsto axb]
\end{equation}
is an isomorphism
of $R$-algebras. Here $A^\op$ denotes the \emph{opposite algebra} of $A$,
which coincides with $A$ as an $R$-module, but with twisted multiplication
$a^\op b^\op=(ba)^\op$. It is clear that $A^\op$ is also an Azumaya $R$-algebra.

The centre $Z(A)$ is equal to $R$ and $\End_R(A)$ is again an Azumaya
$R$-algebra. More generally, if $M$ is a faithful finitely generated projective 
right $R$-module, then $\End_R(M)$ is an Azumaya $R$-algebra.
If $A$ and $B$ are Azumaya $R$-algebras, their tensor product $A\ox_R B$ is
again an Azumaya $R$-algebra.

First's paper \cite{first23} contains a wealth of information about Azumaya 
algebras, with and without involution, 
and we refer to it for a number of definitions and results that we recall
in the remainder of this section.

\begin{prop}[{\cite[Proposition~1.1]{first23}}]\label{prop:Az}
  $A$ is Azumaya over $Z(A)$ and $Z(A)$ is finite \'etale over $R$ if and only
  if $A$ is projective as an $R$-module and separable as an $R$-algebra.
\end{prop}

Since the behaviour of the involution on the centre plays an important role in
the study of algebras with involution, this result helps motivate the following:

\begin{defi}[{\cite[Section~1.4]{first23}}]\label{first-def}  
  We say that $(A,\s)$ is an \emph{Azumaya
  algebra with involution over $R$} if the following conditions hold:
  \begin{itemize}
    \item $A$ is an $R$-algebra with $R$-linear involution $\s$;
    \item $A$ is separable projective over $R$;
    \item the homomorphism $R \rightarrow A$, $r \mapsto r \cdot 1_A$ identifies
      $R$ with $\Sym(Z(A),\s)$.
  \end{itemize}
\end{defi}

Let $(A,\s)$ be an Azumaya algebra with involution over $R$.
Note that $A$ is Azumaya over $Z(A)$ by Proposition~\ref{prop:Az}, 
but may not be  Azumaya over $R$.
Indeed, ``Azumaya algebra with involution'' means ``Azumaya 
algebra-with-involution'' rather than ``Azumaya-algebra with involution''.

The following lemma makes the connection between Definition~\ref{first-def}
and a different definition of Azumaya algebra with involution
(the first sentence of Lemma~\ref{twodefs})
that is introduced in \cite[Section~4]{OP01}:

\begin{lemma}\label{twodefs}
  Let $A$ be an $R$-algebra with $R$-linear involution such that 
  $A$ is an Azumaya algebra 
  over $Z(A)$, $Z(A)$ is $R$ or a quadratic \'etale extension of 
  $R$, and $R = \Sym(Z(A),\s)$. Then $(A,\s)$ is an Azumaya algebra with involution
  over $R$.
  
  The converse holds if $R$ is connected.
\end{lemma}

\begin{proof}
  The first statement is a direct consequence of 
  Proposition~\ref{prop:Az}.   
  For the converse, assume that $R$ is connected and that $(A,\s)$ is an 
  Azumaya algebra with involution over $R$. 
  Then $A$ is Azumaya over $Z(A)$ by Proposition~\ref{prop:Az}, and  
  $Z(A)$ is $R$ or a
  quadratic \'etale extension of $R$ by \cite[Proposition~1.21]{first23}.
\end{proof}

\begin{rem}
  If $R=F$ is 
  actually a field, then  $(A,\s)$ is 
  an Azumaya algebra with involution over $R$ if and only if it
  is a \emph{central simple $F$-algebra with involution} in the sense 
  of \cite[Sections~2.A and 2.B]{BOI}.    
\end{rem}

We recall the following, proved in \cite[second paragraph of Section~1.4]{first23}:
\begin{prop}[Change of base ring]\label{change-of-base}
  Let $T$ be a commutative $R$-algebra. Then  $Z(A \ox_R T) = Z(A) \ox_R T$
  and $(A \ox_R T, \s \ox \Id)$ is an 
  Azumaya algebra with
  involution over $T$ (and in particular $Z(A \ox_R T) \cap
  \Sym(A \ox_R T, \s \ox \id) = T$).
\end{prop}

\begin{cor}\label{change-of-base-bis}
  Let $T$ be a commutative $R$-algebra which is also a field. Then either $Z(A
  \ox_R T)=T$ or $Z(A \ox_R T)$ is a quadratic \'etale extension of $T$, and
  $(A\ox_R T, \s\ox\Id)$ is a central simple $T$-algebra with involution.
\end{cor}

\begin{rem}\label{rem-first-0}
  Let $S$ be a quadratic \'etale $R$-algebra with standard
  involution $\vt$. 
  If $T$ is a commutative $R$-algebra,
  then $S \ox_R T$ is a quadratic \'etale $T$-algebra, cf. 
  \cite[Propositions~2.3.4 and 9.2.5]{ford17}.

  Moreover, $(S,\vt)$ is an Azumaya algebra with involution over $R$ in
  the sense of \cite[Section~4]{OP01} (i.e., it has the properties listed
  in the first sentence of Lemma~\ref{twodefs}), 
  and therefore an Azumaya algebra with involution
  over $R$ in the sense of \cite{first23} (i.e., in the sense of this paper),
  by Lemma~\ref{twodefs}. Therefore $T = \Sym(S \ox_R T, \vt \ox \id)$
  by Proposition~\ref{change-of-base}, and it follows that $\vt \ox \id$ is
  the standard involution of $S \ox_R T$. Then \eqref{eq:tr}
  applies and yields, for all $x\in S$,
  \[
    \Tr_{S \ox_R T / T}(x \ox 1) =\Tr_{S/R}(x) \ox 1.
  \]
\end{rem}

We collect some results about ideals of $A$:
\begin{prop}\label{first-ideals} The following properties hold:
  \begin{enumerate}
    \item\label{first-ideals1} If $I$ is an ideal of $R$, then $A \ox_R (R/I) \cong A/AI$;
    \item $J(A) = A \cdot J(R)$, where $J$ denotes the Jacobson radical;
    \item If $a \in A$ is such that $a \in (A/A \mathfrak{m})^\x$ for every 
    maximal ideal
      $\mathfrak{m}$ of $R$, then $a \in A^\x$.
  \end{enumerate}
\end{prop}

\begin{proof}
  (1) This is well-known.
  (2) Is \cite[Lemma~1.5]{first23}, since $A$ is a separable projective
      $R$-algebra.
  (3) Is \cite[Lemma~1.6]{first23}, using item \eqref{first-ideals1}.
\end{proof}

\begin{defi}[See {\cite[Section~1.4]{first23}}]
  We say that the involution $\s$ on $A$ 
  is \emph{of orthogonal, symplectic or unitary type at 
  $\fp\in \Spec R$} if 
  $(A \ox_R \Qf (R/\fp), \s \ox \id)$ is a central simple algebra with 
  involution of orthogonal, symplectic or unitary type, respectively,  
  over the quotient field $\Qf 
  (R/\fp)$, and that
  $\s$ is \emph{of orthogonal, symplectic or unitary type} 
  if it is  of orthogonal, symplectic or unitary type, respectively, 
   at all $\fp\in \Spec R$.
\end{defi}

We often simply say that $\s$ is orthogonal, symplectic or unitary (at $\fp$).
We refer to \cite{BOI} for the definitions of orthogonal,  symplectic and
unitary involutions on central simple algebras.

\begin{prop}[See {\cite[Proposition~1.21]{first23}}  for a more detailed 
  statement]\label{rem:quad-et}
  Assume that $R$ is connected. Then precisely 
  one of the following  holds:
  \begin{enumerate}
    \item $\s$ is of orthogonal type, 
    $Z(A)=R$, and $\s|_{Z(A)}=\Id_{Z(A)}$;
    \item $\s$ is of symplectic type, 
    $Z(A)=R$, and $\s|_{Z(A)}=\Id_{Z(A)}$;
    \item $\s$ is of unitary type, $Z(A)$ is a quadratic \'etale $R$-algebra,
      and $\s|_{Z(A)}$ is the standard involution on $Z(A)$.
  \end{enumerate}
\end{prop}

\begin{rem}\label{rem:csa-type}
  If $(A,\s)$ is a central simple algebra over a field, the involution $\s$
  is called \emph{of the first kind} if $\s|_{Z(A)}=\id_{Z(A)}$ and 
  \emph{of the second kind} otherwise (so that ``second kind'' is the same
  as ``unitary type''), cf. \cite{BOI}.  
\end{rem}

\begin{lemma}\label{S-diag}
  Let $R$ be semilocal connected, let $T$ be a quadratic \'etale  $R$-algebra,
  and let $\tau$ be an $R$-linear involution on $T$. If $T$ is connected, then
  every nonsingular  hermitian form over $(T,\tau)$ is diagonalizable.
\end{lemma}

\begin{proof}
  The algebra $T$ is semilocal by \cite[VI, Proposition~1.1.1]{knus91}. 
  Since $T$ is also connected,  every projective $T$-module is free 
  by \cite{hinohara}. Furthermore, $2 \in T^\x$.

  If $\tau = \id$, the result follows since every quadratic form over $T$ is
  diagonalizable (see \cite[Proposition~3.4]{baeza}).

  If $\tau \not = \id$. Then $\tau$ is the standard involution on
  $T$ by \cite[Lemma~1.17]{first23}, and $(T,\tau)$ is an Azumaya algebra with
  involution over $R$, cf. Remark~\ref{rem-first-0}. 
  Since  $Z(T) \not = R$, and in particular $\tau$ is not symplectic, 
  we may conclude by
  \cite[Proposition~2.13]{first23}.
\end{proof}

\begin{thm}[Hermitian Morita equivalence]\label{thm:HME}
  Let $R$ be semilocal connected and let $(B,\tau)$ be an Azumaya algebra
  with involution over $R$ such that $A$ and $B$ are Brauer equivalent and
  $\s|_S=\tau|_S$, where $S:=Z(A)=Z(B)$. Assume that $S$ is connected.
  Then there exists $\delta\in\{-1,1\}$
  such that the  
  categories
  $\Herm^\ve(A,\s)$ and $\Herm^{\delta\ve}(B,\tau)$ are equivalent, 
  and this equivalence induces an
  isomorphism of Witt groups $W^{\ve}(A,\s)\cong 
  W^{\delta\ve}(B,\tau)$. Specifically, if $\s|_S=\id_S$, then
  $\delta=1$ if $\s$ and $\tau$ are both orthogonal or both symplectic and
  $\delta=-1$ otherwise; if $\s|_S\not=\id_S$, then $\delta$ can be chosen
  freely in $\{-1,1\}$. 
\end{thm}

\begin{proof}
  If $\s|_S\not=\id_S$,  then $S$ is a quadratic \'etale
  $R$-algebra by Proposition~\ref{rem:quad-et}
  and is thus semilocal by \cite[VI, Proposition~1.1.1]{knus91}. 
  
  Since $A$ and $B$ are Brauer equivalent, there exist faithful
  finitely generated 
  projective $S$-modules $P$ and $Q$ such that $A\ox_S \End_S(P)\cong
  B\ox_S \End_S(Q)$, cf. \cite[III, (5.3)]{knus91}. Since $S$
  is semilocal connected, $P$ and $Q$ are free  
  by \cite{hinohara}. It 
  follows that
  there exist $k,\ell \in \N$ such that 
  \begin{equation}\label{end-iso}
   \End_A(A^k)\cong \End_B(B^\ell). 
  \end{equation}
  
  Consider the nonsingular hermitian forms $(A^k, \vf:=k\x \qf{1}_\s)$ and
  $(B^\ell, \psi:=\ell\x \qf{1}_\tau)$ over $(A,\s)$ and $(B,\tau)$, 
  respectively.   
  Since $\vf$ and $\psi$ are hermitian, the involutions $\s$ and $\ad_\vf$
  are of the same type and the same is true for the involutions $\tau$
  and $\ad_\psi$, cf. \cite[Proposition~2.11]{first23}. In particular,
  $\s|_S=\ad_\vf|_S$ and $\tau|_S=\ad_\psi|_S$.

  The categories 
  $\Herm^\ve(\End_A(A^k), \ad_\vf)$ and $\Herm^\ve(A,\s)$ are equivalent by
  hermitian Morita theory. The same is true for
  the categories
  $\Herm^\ve(\End_B(B^\ell),\ad_\psi)$ and $\Herm^\ve(B,\tau)$. 
  Furthermore, these equivalences respect orthogonal sums and send
  hyperbolic spaces to hyperbolic spaces, cf. \cite[I, 
  Theorem~9.3.5]{knus91}. 
  
  The involution $\ad_\vf$ induces an involution $\omega$
  on $\End_B(B^\ell)$ of the same type via the
  isomorphism \eqref{end-iso}, and thus $\omega|_S =\ad_\psi|_S$ since
  $\s|_S=\tau|_S$.  
  By the Skolem-Noether theorem 
  \cite[Theorem~8.6]{BFP}, $\omega$ differs from $\ad_\psi$ by an inner 
  automorphism: there exists $\delta \in \{-1,1\}$ and a unit $u\in 
  \End_B(B^\ell)$ with 
  $\ad_\psi(u)=\delta u$ such that $\omega = \Int(u)\circ \ad_\psi$,
  where $\delta$ can be freely chosen in $\{-1,1\}$ if $\s|_S \not = \id_S$.
  Moreover, $\delta$ is unique if $\s|_S=\id_S (=\ad_\psi|_S)$:  if 
  $\Int(u)=\Int(u')$,
  then $u=u's$ for some $s\in S^\x$, and it follows that
  $\ad_\psi(u)=\delta u$ if and only if $\ad_\psi(u')=\delta u'$.
  
  By \cite[I, (5.8)]{knus91} (see also \cite[Section~2.7]{first23})
  there is an equivalence between
  the categories $\Herm^\ve(\End_B(B^\ell),\omega)$ and 
  $\Herm^{\delta\ve}(\End_B(B^\ell),\ad_\psi)$ 
  (and an induced equivalence between the categories 
  $\bH^\ve(\End_B(B^\ell),\omega)$ and 
  $\bH^{\delta\ve}(\End_B(B^\ell),\ad_\psi)$)
  which respects orthogonal
  sums and hyperbolic spaces. The equivalence of the categories
  $\Herm^\ve(A,\s)$ and $\Herm^{\delta\ve}(B,\tau)$ and the isomorphism
  of the Witt groups $W^{\ve}(A,\s)$ and   $W^{\delta\ve}(B,\tau)$ follows.

  It remains to show the claim about $\delta$ when $\s|_S=\id_S$.
  Since $\delta$ is unique,
  it suffices to check its value at any $\fp \in \Spec R$ (i.e., after tensoring 
  over
  $R$ by $\Qf(R/\fp)$). By \cite[Proposition~2.7]{BOI} we know that $\delta = 1$
  if $\s$ and $\tau$ are both orthogonal or both symplectic at $\fp$ and $\delta
  = -1$ otherwise. The statement then follows, using
  Proposition~\ref{rem:quad-et}.
\end{proof}

\begin{rem}\label{rem:notcon}
  Let $R$ be connected and let $(A,\s)$ be an Azumaya algebra with involution
  over $R$.
  Note that when $Z(A)$ is not connected, then $Z(A)\cong R\x R$ by 
  Section~\ref{sec:quadalg} and Lemma~\ref{twodefs}. Hence,
  $W^\ve(A,\s)=0$, cf. 
  \cite[Example~2.4]{first23}.
\end{rem}

\subsection{Orderings, Spaces of signatures}
\label{sec:ord-sig}

Let $R$ be a commutative ring, and let $(A,\s)$ be an Azumaya algebra
with involution over $R$.
The real spectrum of $R$, $\Sper R$, is the set of all orderings on $R$, cf.
\cite[Definitions~3.3.1(a) and 3.3.4, and Proposition~3.3.5]{KSU}. It is
equipped with the Harrison topology, with subbasis given by the sets of the form
\[
 \mathring H(r) := 
 \{\alpha \in \Sper R \mid r > 0 \text{ at } \alpha\} 
    = \{\alpha \in \Sper R \mid r \in \alpha \setminus -\alpha\},
\]
for all $r \in R$.

Let $\alpha \in\Sper R$. We often write $x\geq_\alpha y$ for $x-y\in\alpha$.
The support of $\alpha$ is the prime ideal $\Supp(\alpha) := \alpha \cap
-\alpha \in \Spec R$, and   
we denote by $\bar \alpha$ the ordering induced by $\alpha$ on
the quotient field
$\kappa(\alpha):=\Qf (R/\Supp(\alpha))$, and by $k(\alpha)$ a real closure
of $\kappa(\alpha)$ at $\bar\alpha$. Observe that 
$\kappa(\bar\alpha):=\Qf (\kappa(\alpha)  /\Supp(\bar\alpha))
=\Qf (\kappa(\alpha)  /\{0\})=\kappa(\alpha)$, and that we can take 
$k(\bar\alpha)=k(\alpha)$.

We also define
\[(A(\alpha), \sigma(\alpha)) := (A \ox_R \kappa(\alpha), \s \ox \id),\]
which is a central simple $\kappa(\alpha)$-algebra with involution by
Corollary~\ref{change-of-base-bis}.

\begin{rem}\label{M-ox}
  Let $M$ be an $R$-module and let $\alpha \in \Sper R$. Then, if $x \in M
  \ox_R \kappa(\alpha)$, there are $t \in \N$, $m_i \in M$, $s_i \in
  R$ and $r_i \in R \setminus \Supp (\alpha)$ such that
  \[
    x = \sum_{i=1}^t m_i \ox \dfrac{\bar s_i}{\bar r_i} 
      = \sum_{i=1}^t m_i \ox \dfrac{\bar s'_i}{\bar r'}
      = \Bigl(\sum_{i=1}^t m_i s'_i \Bigr) \ox \dfrac{1}{\bar r'},
  \]
  for some  $s'_i \in R$, $r' \in R \setminus \Supp (\alpha)$.
  Therefore
  \[M \ox_R \kappa(\alpha) = \Big\{m \ox \dfrac{1}{\bar r} \,\Big|\, m 
  \in M,\ r \in
  R \setminus \Supp (\alpha)\Big\}.\]
\end{rem}

We denote by $\Sign R$ the set of signatures of $R$, i.e., the space of all
morphisms of rings from the Witt ring $W(R)$ to $\Z$. We recall some facts from
\cite[Section 5 up to p.~89]{knebusch84}:

\begin{itemize}
  \item The set $\Sign R$ is equipped with the coarsest topology that makes all maps
    \[\sign_\bullet \vf \colon  \Sper R \rightarrow \Z,\ \alpha \mapsto \sign_\alpha \vf\]
    continuous, for $\vf \in W(R)$. When $R$ is semilocal, a basis for this
    topology is given by the sets
    \[H(u_1,\ldots,u_k) := \{\tau \in \Sign R \mid \tau(u_1) = \cdots = 
    \tau(u_k) =1\},
    \]
    for $k \in \N$ and $u_1, \ldots, u_k \in R^\x$.
  \item Denoting by $\Sperm R$ the space of all elements of $\Sper R$
    that are maximal for inclusion
    (equipped with the induced topology), the natural map
    \[\Sperm R \rightarrow \Sign R,\ \alpha \mapsto \sign_\alpha,\]
    is continuous and surjective. If $R$ is semilocal, this map is a
    homeomorphism  and we identify $\Sign R$ and $\Sperm R$.
    In general, we have continuous maps 
    \begin{equation*}
      \Sperm R \subseteq \Sper R \stackrel{\xi}{\rightarrow}
      \Spec R,
    \end{equation*}
    where $\xi$ is defined by $\xi(\alpha) := \Supp(\alpha)$. 
    (Note that if $R=F$ is a field, then $\Sperm R=\Sper R = X_F$,
    the space of orderings of $F$.)
\end{itemize}
\medskip

The following theorems, due to Knebusch, explain how to obtain the 
maximal ordering associated to a signature on a semilocal ring. These results
can be found in \cite[Theorem~4.8]{kne75} and 
\cite[pp.~87-88]{knebusch84}. 
The connection with maximal orderings is given in the second reference,
but is presented for connected rings. It is pointed out that this assumption 
can be made without loss of generality, but we quickly present an argument
in the proof below.

\begin{thm}
  Assume that $R$ is semilocal and let $s$ be a signature on $R$. Define
  \[Q(s):= 
  \Big\{r_1^2u_1 + \cdots + r_k^2u_k \,\Big|\, k \in \N,\ u_i
  \in R^\x,\ s(\qf{u_i})=1, \ \sum_{i=1}^k Rr_i  = R\Big\}.\]
  Then
  \begin{enumerate}
    \item $Q(s)$ is closed under sum, $Q(s) \cap -Q(s) = \emptyset$ and $p(s) := R
      \setminus (Q(s) \cup -Q(s))$ is a prime ideal of $R$.
    \item $\alpha(s) := Q(s) \cup p(s)$ is a maximal ordering on $R$ with
      support $p(s)$ such that
      $\sign_{\alpha(s)} = s$.
  \end{enumerate}
\end{thm}
\begin{proof}
  For (1), see \cite[Theorem~4.8]{kne75}. 
  For (2), it follows immediately from the properties of $Q(s)$ that
  $\alpha(s)$ is an ordering on $R$ with support $p(s)$ and that
  $\sign_{\alpha(s)} = s$. 
  Suppose that $\alpha(s)$ is not maximal, so that there is
  $\beta \in \Sper R$ such that $\alpha(s) \subsetneq \beta$. Take $x \in \beta
  \setminus \alpha(s)$. Since $x \not \in \alpha(s)$ we have $x \in -\alpha(s) \setminus \alpha(s) = -Q(s)$ and
  thus $x = -(r_1^2u_1 + \cdots + r_k^2u_k)$ as described above. In $\Qf(R/\Supp
  \beta)$ we have $\bar x = -\sum_{i=1}^k \bar r_i^2 \bar u_i$ with $\bar u_i \in \bar
  \beta$, $\bar u_i \not =0$, so that $\bar u_i >_{\bar \beta} 0$. In particular $\bar x \in -\bar \beta$. Since $\bar x
  \in \bar \beta$ by choice of $x$ we have $\bar x = 0$, which implies that
  $\bar r_i = 0$ for every $i$, i.e., $r_1, \ldots, r_k \in \Supp \beta$,
  contradicting $\sum_{i=1}^k Rr_i = R$.
\end{proof}
Applying this result to $s = \sign_\alpha$ for $\alpha \in \Sperm R$, we obtain
the following theorem.

\begin{thm}\label{alpha-semilocal}
  Assume that $R$ is semilocal and that $\alpha \in \Sperm R$. Then
  \[\alpha \setminus \Supp(\alpha) = 
  \Big\{r_1^2u_1 + \cdots + r_k^2u_k \,\Big|\, k \in \N,\ u_i
  \in \alpha \cap R^\x, \ \sum_{i=1}^k Rr_i  = R\Big\}.\]
  The special case of $R$ local easily follows, but is worth noting and was
  obtained earlier (without the link to $\Sper R$, which was introduced later),
  first in \cite{KK72} if $2 \in R^\x$, and then in general in \cite{kne75}:
  \[\alpha \setminus \Supp(\alpha) = \{u_1 + \cdots + u_k \mid k \in \N,\ u_i
  \in \alpha \cap R^\x\}.\]
\end{thm}

Finally, we mention the following simple fact for future use: 

\begin{lemma}\label{cofinal}
  Let $F$ be a field with space of orderings $X_F$. Let $P\in X_F$, and denote 
  by $F_P$ any real closure of $F$ at $P$.  
  Then $F$ is cofinal in $F_P$, i.e., 
  for every $a\in F_P$
  there exists $b\in F$ such that $a\leq b$. In particular, for every
  $a\in F_P$, if $a>0$, then there exists $b\in F$ such that $0<b\leq a$.    
\end{lemma}

\begin{proof}
  Let $p(X)=a_0+a_1X + \cdots + a_{k-1}X^{k-1} + X^k \in F[X]$ be such that 
  $p(a)=0$.
  Then $a \leq \max\{1, |a_0| + \cdots + |a_{k-1}|\}$ by 
  \cite[Proposition~1.7.1]{KSU}. The second statement follows by taking
  inverses.
\end{proof}

\subsection{Positive definite matrices}

  The results in this section are well-known, but we could not find
  a reference for the quaternion case. 
  In this section we assume that $F$ is a real closed field and
  that $(D,\vt) \in \{(F,\id), (F(\sqrt{-1}),\gamma), ((-1,-1)_F, \gamma)\}$,
  where $\gamma$ denotes the canonical involution in each case. We 
  consider norms with values in $F$. For vectors in $D^k$
  we consider the euclidean norm, i.e., for $X=(x_1,\ldots, x_k)^t \in D^k$ we
  have
  $\|X\|= \sqrt{\sum_{i=1}^k n(x_i)}$,
  where $n(x) := \vt(x)x$. On $M_k(D)$ we use the induced
  operator norm,  i.e.,
  $\|M\|_\op :=\sup_{\|X\| = 1} \|MX\|$.
  Tarski's transfer principle \cite[Corollary~11.5.4]{mar08} ensures 
  that this supremum exists and that the operator norm
  is equivalent to the maximum norm determined by the unique ordering of $F$
  on the $F$-vector space $M_k(D)$ (both properties can be expressed by
  first-order formulas in the language of ordered fields, that are true in $\R$). Therefore, $ \|\cdot \|_\op $ defines
  the same topology as the one induced by the ordering of $F$.

\begin{lemma}\label{PD-open}
  Let $k \in \N$. Then 
  \[\PD_k(D,\vt) := \{B \in \Sym(M_k(D), \vt^t) \mid \vt(X)^t B X > 0 \text{ for
  every } X \in D^k\setminus\{0\}\}\]
  is an open subset of $\Sym (M_k(D), \vt^t)$ for the topology induced by the
  unique ordering of $F$. 
\end{lemma}
\begin{proof}
  Since this property can be expressed by a first-order formula in the
  language of fields in each of the three cases
  $(D,\vt) \in \{(F,\id), (F(\sqrt{-1}),\gamma), ((-1,-1)_F, \gamma)\}$, 
  it suffices to prove it for $F
  = \R$ by Tarski's transfer principle \cite[Corollary~11.5.4]{mar08}. 
  The well-known proof below works in all three cases, and we give
  the details in order to point out that it also works in the quaternion case.
  Observe that
  for every $X, Y \in D^k$, we have
  $\|\vt(X)^tY\| \le \|X\| \cdot \|Y\|$
  (the verification is direct in the quaternion case).

  We reformulate $M \in \PD_k(D,\vt)$ as: there is $\delta > 0$ such that
  $\vt(X)^t M X \ge \delta$ for every $X \in D^k$, $\|X\| = 1$.  Let $M \in 
  \PD_k(D,\vt)$ and let $\delta$ be as described
  above. Let $B \in \Sym(M_k(D), \vt^t)$ be such that $\|B-M\|_\op \le 
  \ve$.
  Then
  \[\vt(X)^t B X = \vt(X)^t M X +(\vt(X)^t B X - \vt(X)^t M X),\]
  and, if $\|X\| = 1$,
  \begin{align*}
    |\vt(X)^t B X - \vt(X)^t M X| &= |\vt(X)^t(B-M)X| = \|\vt(X)^t(B-M)X\| \\
      & \le \|X\| \cdot \|(B-M)X\| \le  \|B-M\|_\op  \le\ve
  \end{align*}
  from which the result follows if we take $\ve = \delta/2$.
\end{proof}

\section{$\CM$-Signatures of hermitian forms}
\label{sec3}

If $(A,\s)$ is a central simple $F$-algebra with involution 
over a field $F$, $h$ is a hermitian form over $(A,\s)$, and $P$ is an
ordering on $F$, we defined in \cite[Section~3.2]{A-U-Kneb} 
the signature of $h$ at
$P$ with respect to a particular Morita equivalence $\CM_P$ 
(which determines the
sign of the signature), denoted $\sign^{\CM_P}_P h$. We only considered 
nonsingular forms in \cite[Section~3.2]{A-U-Kneb}, which was unnecessarily 
restrictive since the method we used (reduction to the Sylvester signature via
scalar extension to a real closure 
of $F$ at $P$ and hermitian Morita theory) 
applies in fact to forms that may be singular.

We will use the notation and
results from \cite{A-U-Kneb}. Note that this signature is also presented in
\cite[second half of p.~499]{A-U-prime}, where the omission in \cite{A-U-Kneb}
of one (irrelevant)
case for $(D_P, \vartheta_P)$ has been rectified.

Let $R$ be a commutative ring (with $2\in R^\x$), let
$(A,\s)$ be an Azumaya algebra with involution over $R$, let $S=Z(A)$
and let $\iota:=\s|_S$.

\begin{defi}\label{def:sig_h}
  Let $h$ be a hermitian form over $(A,\s)$ and let $\alpha \in \Sper R$.  Then
  $h\ox \kappa(\alpha)$ is a hermitian form over the central simple algebra with
  involution $(A(\alpha),\s(\alpha))$ and we define the
  \emph{$\CM$-signature} of $h$ at $\alpha$ by
  \begin{equation}\label{eq:M-sign}
    \sign^{\CM}_\alpha h := \sign^{\CM_{\bar\alpha}}_{\bar \alpha} (h \ox
      \kappa(\alpha)),
  \end{equation}
  where $\CM_{\bar\alpha}$ is a Morita equivalence as in 
  \cite[Section~3.2]{A-U-Kneb}.
\end{defi}

Note that the superscript $\CM$ on the left hand side of \eqref{eq:M-sign}
signifies that each
computation of $\sign^{\CM}_\alpha$ depends on a choice of Morita
equivalence $\CM_{\bar\alpha}$. The use of a different Morita equivalence can
result at most in a change of sign, cf. \cite[Proposition~3.4]{A-U-Kneb}.
Therefore, the notation $\sign^{\CM}_\alpha$ should specify what Morita
equivalence $\CM_{\bar\alpha}$ is used but, in order not to overload the
notation, we assume that for a given $\alpha$ we always use the same Morita
equivalence $\CM_{\bar\alpha}$.

In fact, the main drawback of the $\CM$-signature is that the sign of the
signature of a form can be changed arbitrarily at each ordering by taking a
different Morita equivalence. This is in particular a problem if
we hope to consider the total signature of a form as a continuous 
function on $\Sper R$. We
solved this problem in the case of central simple algebras with involution by
introducing a ``reference form'', that determines the sign of the
signature at each ordering, cf.
\cite[Section~6]{A-U-Kneb} and \cite[Section~3]{A-U-prime}. 
We will show in a forthcoming publication that the
same can be done in the case of Azumaya algebras with involution.

\begin{rem}\label{sign-bounded}  
  Let $(B,\tau)$ be a central simple $F$-algebra with involution, $P \in X_F$, 
  and
  $F_P$ a real closure of $F$ at $P$. Note that $F=\Sym(Z(B),\tau)$.
  Then $B \cong M_n(D)$ and
  $B \ox_F F_P \cong M_{n_P}(D_P)$, where $D$ and $D_P$ are  division algebras
  with involution 
  over $F$ and $F_P$, respectively. (We do not need to name the involutions
  on $D$ and $D_P$ here.) 
  Clearly, $n \le n_P \le \sqrt{\dim_F B}$.
  Let $b \in \Sym(B,\tau)$. By \cite[Proposition~4.4]{A-U-PS}, we have
  $\sign^{\CM_P}_P \qf{b}_\tau \le n_P \le \sqrt{\dim_F B}$.
  Applying this to $(A(\alpha),\s(\alpha))$ in Definition~\ref{def:sig_h},
  let $\mathcal{S}$ be a set of generators of $A$ as an $R$-module, and let 
  $a \in \Sym(A,\s)$.
  We have
  \[\sign^{\CM}_\alpha \qf{a}_\s \le \sqrt{\dim_{\kappa(\alpha)} A(\alpha)}
  \le |\mathcal{S}|,\]
  which provides a bound on $\sign^{\CM}_\alpha \qf{a}_\s$ which is independent
  of $a$ and $\alpha$.
\end{rem}

\begin{prop}\label{mult-sign}
  Let $h$ be a hermitian form over $(A,\s)$, let $q$ be a quadratic form over
  $R$ and let $\alpha \in \Sper R$. Then
  \[\sign^\CM_\alpha (q \ox h) = (\sign_\alpha q) \cdot (\sign^\CM_\alpha h),\]
  where the same Morita equivalence is used in the computation of the signature
  on both sides of the equality.
\end{prop}
\begin{proof}
  By definition $\sign^\CM_\alpha(q \ox h)$ is equal to 
  $\sign^{\CM_{\bar \alpha}}_{\bar \alpha}
  ((q \ox h) \ox \kappa(\alpha))$, where $(q \ox h) \ox \kappa(\alpha)$ is 
  considered as a form over
  $(A(\alpha),\s(\alpha))$. But $\sign^{\CM_{\bar \alpha}}_{\bar \alpha}
  ((q \ox h) \ox \kappa(\alpha))
   = (\sign_{\bar \alpha} q \ox \kappa(\alpha)) \cdot
  (\sign^{\CM_{\bar \alpha}}_{\bar \alpha} h \ox \kappa(\alpha))$ by
  \cite[Proposition~3.6]{A-U-Kneb}. The result follows.
\end{proof}

\begin{lemma}\label{trace-sign-zero}
  Let $T$ be a quadratic \'etale $R$-algebra with standard involution $\vt$.
  Let $h$ be a hermitian form over $(T,\vt)$. Then, for every 
  $\alpha \in \Sper R$,
    $\sign^\CM_\alpha h=0$ implies $\sign_\alpha \Tr_{T/R} (h)=0$.
\end{lemma}

\begin{proof}
  Note 
  that $(T,\vt)$ is an Azumaya algebra with involution over $R$ 
  and that $\vt\ox\id$ is the standard involution on $T\ox_R k(\alpha)$ 
  by Remark~\ref{rem-first-0}. In particular, signatures are defined, and
  by the definition of signatures for central simple algebras with involution,
  the form $h \ox_R k(\alpha)$ has signature $0$ at the unique ordering on
  $k(\alpha)$. Writing $h \ox_R k(\alpha)\simeq \vf \perp 0$,
  where $\vf$ is nonsingular and $0$ is the zero form of appropriate rank,
   cf. \cite[Proposition~A.3]{A-U-PS}, it
  follows that  $\vf$ has signature $0$ at the unique 
  ordering on $k(\alpha)$, and thus is weakly hyperbolic 
  (i.e., $\ell\x\vf$ is hyperbolic for some $\ell\in \N$)
  by 
  \cite[Theorem~4.1]{L-U} (or \cite[Theorem~6.5]{B-U}).

  As recalled in Section~\ref{sec:quadalg}, $\Tr_{T/R}$ is 
  $R$-linear and $\Tr_{T/R}(h)$ is a quadratic form over $R$.
  Also,  $\Tr_{T \ox_R k(\alpha) /   k(\alpha)}$ is an involution
  trace for $\vt\ox \id$, and it follows that $\Tr_{T \ox_R k(\alpha) /
  k(\alpha)}(\vf)$ is weakly hyperbolic by Section~\ref{sec:quadalg},
  and thus has signature~$0$.
  Since $\Tr_{T/R}(h) \ox k(\alpha)
  = \Tr_{T \ox_R k(\alpha) /  k(\alpha)}(h \ox k(\alpha))$ by
  Remark~\ref{rem-first-0}, and observing that
  \[
  \Tr_{T \ox_R k(\alpha) /  k(\alpha)}(h \ox k(\alpha)) \simeq
  \Tr_{T \ox_R k(\alpha) /  k(\alpha)}(\vf)\perp \Tr_{T \ox_R k(\alpha) /  k(\alpha)}(0),
  \]
  it follows that  $\sign_\alpha
  \Tr_{T/R}(h) = 0$.  
\end{proof}

\begin{defi}
  We call
  \[\Nil[A,\s] := \{\alpha \in \Sper R \mid \sign^\CM_\alpha = 0 \}\]
  the set of \emph{nil orderings} of $(A,\s)$. By the observation after 
  Definition~\ref{def:sig_h}, $\Nil[A,\s]$ is independent
  of the choice of Morita equivalence at each $\alpha$.
\end{defi}

\begin{rem}\label{rem:nil}
  If $(A,\s)$ is a central simple $F$-algebra with involution, this definition
  is equivalent to our original one (\cite[Definition~3.7]{A-U-Kneb} by
  \cite[Theorem~6.4]{A-U-Kneb}), from which it follows that if $\s$ is
  orthogonal and $\tau$ is any symplectic involution on $A$, then 
  $\Nil[A,\s]=X_F \setminus \Nil[A,\tau]$.
  
  Also note that if $P\in X_F\setminus \Nil[A,\s]$ and $F_P$ is a real closure 
  of $F$ at $P$, then 
  $(A\ox_F F_P, \s\ox\id)\cong (M_{n_P}(D_P), \Int(\Phi_P)\circ {\vt_P}^t)$,
  where $D_P\in\{F_P, 
  F_P(\sqrt{-1}),(-1,-1)_{F_P}\}$, $\vt_P$ is the canonical involution on
  $D_P$, and $\Phi_P \in \Sym (M_{n_P}(D_P)^\x, {\vt_P}^t)$, cf.
  \cite[p.~4 and Remark~6.2]{A-U-pos}. 
\end{rem}

\begin{lemma}\label{Nil-product}
  Let $\alpha \in \Sper R$. Then statements \eqref{np1} and \eqref{np3} below
  are equivalent:
  \begin{enumerate}
    \item\label{np1} $\alpha \in \Nil[A,\s]$;
    \item\label{np3} $\bar \alpha \in \Nil[A(\alpha),\s(\alpha)]$.
  \end{enumerate}
  Assume in addition that $\s$ is of unitary type at $\Supp(\alpha)$.
  Then \eqref{np1} and \eqref{np3} and the following statements are equivalent:
  \begin{enumerate} \setcounter{enumi}{2}
    \item\label{np3b} $Z(A(\alpha)) \ox_{\kappa(\alpha)} k(\bar\alpha) \cong
      k(\bar\alpha) \x k(\bar\alpha)$, i.e.,
      $Z(A(\alpha)) \ox_{\kappa(\alpha)} k(\alpha) \cong
      k(\alpha) \x k(\alpha)$ since we can take $k(\bar\alpha)=k(\alpha)$
      as observed before;
    \item\label{np2} $Z(A) \ox_R k(\alpha) \cong k(\alpha) \x k(\alpha)$;
    \item\label{np4} $\alpha \in \Nil[S,\iota]$.
  \end{enumerate}
\end{lemma}

\begin{proof}
  \eqref{np1}$\Rightarrow$\eqref{np3}: Assume $\bar \alpha \not \in
  \Nil[A(\alpha),\s(\alpha)]$. Then there exists $z \in \Sym(A(\alpha),
  \s(\alpha))$ such that $\sign^\mab_{\bar \alpha} \qf{z}_{\s(\alpha)}
  \not = 0$ (see
  \cite[Theorem~6.4]{A-U-Kneb}). Using Remark~\ref{M-ox}, write $z = a \ox
  \frac{1}{\bar r}$, for some $r \in R \setminus \Supp (\alpha)$. 
  Since $\bar r$
  is invertible in $\kappa(\alpha)$, we have
  \[
    0\not=\sign^\mab_{\bar \alpha} \qf{z}_{\s(\alpha)} =
    \sign^\mab_{\bar \alpha} \qf{z{\bar r}^2}_{\s(\alpha)}
    = \sign^\mab_{\bar \alpha} \qf{ar\ox 1}_{\s \ox \id_{\kappa(\alpha)}} =
    \sign^\CM_\alpha \qf{ar}_\s,
  \]
  contradicting that $\alpha \in \Nil[A,\s]$.

  \eqref{np3}$\Rightarrow$\eqref{np1}: This follows from the fact that
  $\sign^\CM_\alpha h = \sign^\mab_{\bar \alpha} h \ox
    \kappa(\alpha) = 0$ for any hermitian form $h$  over
  $(A,\s)$.

  For the remaining equivalences, recall that if $(B,\tau)$ is a central simple
  algebra with involution of the second kind over a field $F$, then $P \in
  \Nil[B,\tau]$ if and only if $Z(B \ox F_P) \cong F_P \x F_P$, cf.
  \cite[bottom of p.~499 and Def.~2.1]{A-U-prime}.

  The equivalence \eqref{np3}$\Leftrightarrow$\eqref{np3b} follows from this
  observation, and we have \eqref{np3b}$\Leftrightarrow$\eqref{np2} since
  $k(\alpha)$ is a real closure of $\kappa(\alpha)$ at $\bar \alpha$
  and $Z(A(\alpha))=Z(A)\ox_R \kappa(\alpha)$ (by 
  Proposition~\ref{change-of-base}). Finally,
  the equivalence \eqref{np2}$\Leftrightarrow$\eqref{np4} follows from
  \eqref{np1}$\Leftrightarrow$\eqref{np2} applied to $(S, \iota)$ since
  $S=Z(S)=Z(A)$.
\end{proof}

\begin{cor}\label{nil-open-second-kind}
  Assume that $R$ is semilocal and that $\s$ is of unitary type.
  Then there is $d \in R^\x$ such that $\Nil[A,\s] = \mathring H(d)$.
\end{cor}

\begin{proof}
  By \cite[Proposition~1.21]{first23}, $S=Z(A)$ is a quadratic \'etale
  $R$-algebra, and for every $\alpha \in \Sper R$, $\s(\alpha)$ is of 
  unitary type at the prime ideal $\Supp(\alpha)$. 
  Furthermore, as recalled in Section~\ref{sec:quadalg}, $S = R \oplus
  \lambda R$ for some $\lambda \in S$ such that $d:=\lambda^2 \in R^\x$.  In particular, for any $\alpha \in \Sper R$, 
  $Z(A(\alpha)) = S \ox_R \kappa(\alpha)
  = \kappa(\alpha) \oplus (\lambda \ox 1) \kappa(\alpha) =
  \kappa(\alpha)(\sqrt{\bar d})$, where $\bar d$ is the image of $d$ in
  $R/\Supp(\alpha)$, and where the first equality follows from 
  Proposition~\ref{change-of-base}.
  Therefore, for $\alpha \in \Sper R$:
  \begin{align*}
    \alpha \in \Nil[A,\s] 
      &\Leftrightarrow Z(A(\alpha))\ox_{\kappa(\alpha)} k(\bar\alpha)
      \cong k(\bar\alpha)\x k(\bar\alpha) \text{ by Lemma~\ref{Nil-product}} \\
      &\Leftrightarrow \kappa(\alpha)(\sqrt{\bar d})\ox_{\kappa(\alpha)}
      k(\bar\alpha)
      \cong k(\bar\alpha)\x k(\bar\alpha)\\      
      &\Leftrightarrow \sqrt{\bar d} \in k(\bar\alpha)
      \Leftrightarrow \bar d \in \bar \alpha 
      \Leftrightarrow d \in \alpha\\
      &\Leftrightarrow \alpha\in \mathring H(d) \ \text{since } d 
      \text{ is invertible}.\qedhere
  \end{align*}
\end{proof}

\begin{rem}\label{S-connected}
  Recall from Section~\ref{sec:quadalg} that when $R$ is connected, then either
  $S$ is connected or $S \cong R \x R$. Moreover, if
  $\Sper R \setminus \Nil[A,\s]\not=\emptyset$, we cannot have $S \cong R \x R$
  by Lemma~\ref{Nil-product}. 
  Therefore, if $R$ is connected and $\Sper R \setminus 
  \Nil[A,\s]\not=\emptyset$, then $S$ is also connected.
\end{rem}

\begin{lemma}\label{skew-inv-csa-no-pc}
  Let $F$ be a field and let $(B,\tau)$ be a central simple $F$-algebra
  with involution of the first kind. If $\Nil[B,\tau] \not = \emptyset$, then
  $\deg B$ is even. 
\end{lemma}

\begin{proof}
  Assume that $\deg B$ is odd. Then $B$ is split and $\tau$ is orthogonal
  by \cite[Corollary~2.8(1)]{BOI}. It follows that $\Nil[B,\tau] = \emptyset$ 
  by \cite[Definition~3.7]{A-U-Kneb}.
\end{proof}

\begin{lemma}\label{skew-inv-no-pc} 
  Assume that $R$ is semilocal 
  connected and that $\s$ is of orthogonal or symplectic
  type. If $\Nil[A,\s] \not = \emptyset$, then
  $\Skew(A^\x,\s) \not = \emptyset$.
\end{lemma}
\begin{proof}
  Let $\alpha \in \Nil[A,\s]$. Then $\bar\alpha\in \Nil[A(\alpha), \s(\alpha)]$
  by Lemma~\ref{Nil-product}. Hence,  $\deg A \ox_R
  \kappa(\alpha)$ is even by Lemma~\ref{skew-inv-csa-no-pc}. 
  Since $R$ is connected, the rank of $A$ is constant
  (since $A$ is a projective $R$-module, see \cite[p.~12]{Saltman99}). It
  follows that $\deg A \ox_R \Qf(R/\fp) $ is even for every 
  $\fp \in \Spec R$. We can then apply \cite[Lemma~1.26]{first23} with 
  $\varepsilon = -1$.
\end{proof}

\subsection{Elements of maximal signature}
In \cite{A-U-pos}, working with central simple algebras with involution, we
investigated the maximal value that the signature at $P \in X_F$ can take (when it is
non-zero) when applied to one-dimensional nonsingular forms. We found that this
maximal value is the matrix size of the algebra over its skew-field part
after scalar extension to $F_P$ 
(and linked it to the existence of
positive involutions), cf. \cite[Proposition~6.7, Theorem~6.8]{A-U-pos}.

We are interested in the same question when $(A,\s)$ is an Azumaya algebra with
involution. More precisely, we will show in Corollary~\ref{same-max}
that if $\alpha\not\in\Nil[A,\s]$, then
this maximal 
value  
is the matrix size $n_{\bar \alpha}$
of  $A\ox_R k(\alpha)$ over its skew-field part 
(i.e., $A\ox_R k(\alpha)\cong M_{n_{\bar \alpha}}(D_{\bar \alpha})$ using the
notation from Remark~\ref{sign-bounded}). We first introduce some notation:

\begin{defi}
  If $(B,\tau)$ is an Azumaya algebra with involution over $R$, and
  $\alpha \in \Sper R$, we define:
  \begin{align*}
    m_\alpha(B,\tau) &:= \max\{\sign^\CM_\alpha \qf{b}_\tau \mid b \in
    \Sym(B^\x,\tau)\}\\
    \intertext{and}
    M^\CM_\alpha(B,\tau) &:= \{b \in \Sym(B^\x,\tau) \mid \sign^\CM_\alpha
      \qf{b}_\tau = m_\alpha(B,\tau)\}.
  \end{align*}
Observe that $m_\alpha(B,\tau)$ is independent
  of the choice of the Morita equivalence $\CM_{\bar\alpha}$ 
  (cf. Definition~\ref{def:sig_h}), 
  and  is finite by Remark~\ref{sign-bounded}. 
\end{defi}

We introduce some notation that will be used in the next four results. For
$\alpha \in \Sper R$, define
\begin{align*}
  \CCS_\alpha(A,\s) := \bigcup\{D_{(A,\s)} \qf{a_1,\ldots, a_k}_\s \mid k \in 
  \N,\ a_i \in &\Sym(A,\s),\\
  &a_i \ox 1 \in  M^{\CM_{\bar\alpha}}_{\bar \alpha}(A(\alpha), \s(\alpha))\}.
\end{align*}

Furthermore, for $\fp \in \Spec R$, we
denote by $\pi_\fp$ the canonical projection from $R$ to $R/\fp$ and by
$\pi_{A\fp}$ the canonical projection from $A$ to $A/A\fp$. Then, denoting by
$\bar \s$ the involution induced by $\s$ on $A/A\fp$, we define
\begin{align*}
  \CCS'_\fp := \bigcup \{D_{(A/A\fp, \bar \s)} \qf{\pi_{A\fp}(a_1), \ldots,
  \pi_{A\fp}(a_k)}_{\bar \s} \mid k \in \N,\ &a_i \in \Sym(A,\s),\\
  & a_i \ox 1 \in
  M^{\CM_{\bar\alpha}}_{\bar \alpha}(A(\alpha),\s(\alpha))\}.
\end{align*}

\begin{lemma}\label{sl-case-part1}
  Let $\alpha \in \Sper R$.  
  Then there is an element $b \in \Sym(A,\s)$ such that $b \ox 1 \in
  M^{\CM_{\bar\alpha}}_{\bar \alpha}(A(\alpha), \s(\alpha))$. In particular $b \in \CCS_\alpha(A,\s)$.
\end{lemma}
\begin{proof}
  Let $c \in \Sym(A(\alpha)^\x, \s(\alpha))$ be such that 
  $\sign^{\CM_{\bar
  \alpha}}_{\bar
  \alpha} \qf{c}_{\s(\alpha)} = m_{\bar \alpha}(A(\alpha), \s(\alpha))$. By
  Remark~\ref{M-ox}, there are $b_0 \in A$ and $r \in R \setminus \Supp(\alpha)$
  such that $c = b_0 \ox \frac{1}{\bar r}$.  Let
  $b_1:=\frac{1}{2}(b_0+\s(b_0))$. Then $b_1 \in \Sym(A,\s)$ and
  \[b_1 \ox \frac{1}{\bar r} = \frac{1}{2}(b_0 \ox \frac{1}{\bar r} + \s(b_0)
  \ox \frac{1}{\bar r}) = \frac{1}{2}(c+ \s(\alpha)(c)) = c.\]

  We take $b := rb_1$. Then $b \in \Sym(A,\s)$ and, since $\bar r \in
  \kappa(\alpha)^\x$, we have $b \ox 1 = \bar r^2 c \in
  A(\alpha)^\x$, and
  \[\sign^{\CM_{\bar
  \alpha}}_{\bar \alpha}
  \qf{b \ox 1}_{\s(\alpha)} = \sign^{\CM_{\bar
  \alpha}}_{\bar \alpha}
  \qf{\bar r^2 c}_{\s(\alpha)} = m_{\bar \alpha}(A(\alpha), \s(\alpha)),\]
  so that $b \ox 1 \in M^{\CM_{\bar\alpha}}_{\bar \alpha}(A(\alpha), \s(\alpha))$.
\end{proof}

\begin{lemma}\label{with-cofinal}
  Assume that $R = F$ is a field and let $P\in X_F\setminus\Nil[A,\s]$. 
  Let $a \in \Sym(A,\s)$ be such that
  $\sign^\CM_P \qf{a}_\s=m_P(A,\s)$. Then:
  \begin{enumerate}
    \item $a$ is invertible in $A$;
    
    \item There is $\mu \in P \setminus \{0\}$
  such that $\sign^\CM_P \qf{a-r}_\s=m_P(A,\s)$
  for every $r \in P$
  such that $r \le_P \mu$.  

  \end{enumerate}
\end{lemma}

\begin{proof}
  By \cite[Proposition~6.7]{A-U-pos} we have $m_P(A,\s)=n_P$. 
  We use the notation from Remark~\ref{rem:nil}.
  Since
  $P \not \in \Nil[A,\s]$, it follows from the computation of
  $\CM$-signatures (cf. the beginning of Section~\ref{sec3}) 
  that $\sign^\CM_P \qf{a}_\s$ is equal
  to the Sylvester signature of the form 
  $\qf{\Phi_P^{-1} (a \ox 1)}_{{\vt_P}^t}$, where $\Phi_P^{-1} (a \ox 1) \in 
  \Sym(M_{n_P}(D_P), {\vt_P}^t)$. Since $D_P
  \in \{F_P, F_P(\sqrt{-1}), (-1,-1)_{F_P}\}$ and $\vt_P$ is the canonical
  involution on $D_P$, the matrix $\Phi_P^{-1} (a \ox 1)$ can be diagonalized
  by congruences (which does not change the Sylvester signature), so we can
  assume that $\Phi_P^{-1} (a \ox 1)$ is diagonal. Since it has Sylvester
  signature $n_P$, its diagonal elements are all positive, i.e., $\Phi_P^{-1}
  (a \ox 1) \in \PD_{n_P}(D_P, \vt_P)$. In particular, $\Phi_P^{-1} (a \ox 1)$ 
  is invertible. Therefore $a \ox 1$
  is not a zero divisor, and $a$ is not a zero divisor in $A$. It follows that
  $a$ is invertible since $A$ is Artinian. This proves (1).
  
  For (2): The element $a\ox 1$ is in $\Phi_P\cdot
  \PD_{n_P}(D_P, \vt_P)$, which is an open subset
  of $\Sym (M_{n_P}(D_P), {\vt_P}^t)$ by Lemma~\ref{PD-open}.
  It follows that there is $\ve>0$ in $F_P$ such that for all
  $M\in \Sym (M_{n_P}(D_P),
  {\vt_P}^t)$ that satisfy $\| M\|_\op<\ve$ we have
  $a\ox 1-M\in \Phi_P\cdot \PD_{n_P}(D_P, \vt_P)$.
  Taking $\mu\in F_P$ such that $0<\mu < \ve$, we obtain
  $\|\mu I_{n_P}\|_\op<\ve$. Hence,
  $a\ox 1-\mu I_{n_P} \in \Phi_P\cdot \PD_{n_P}(D_P, \vt_P)$.
  In particular, the signature of the form $\qf{a\ox 1 -\mu I_{n_P}}_{\s\ox\id}$
  equals $n_P$.

  However, $F$ is cofinal in
  $F_P$ by Lemma~\ref{cofinal}, so we can find such a $\mu$ in $F$. The
  choice of $\mu$
  guarantees
  that $\sign^\CM_P \qf{a-r}_\s=n_P=m_P(A,\s)$
  whenever $r \in P$, $r \le_P \mu$.  
\end{proof}

\begin{lemma}\label{improvement}
  Let $\fm$ be a maximal ideal of $R$ and let $\alpha \in \Sper R$ be such that
  $\Supp \alpha \subseteq \fm$. Assume that
  \begin{equation}\label{property-Am1}
    \forall a \in \Sym(A,\s) \quad a \ox_R 1_{\kappa(\alpha)} \in M^{\CM_{\bar\alpha}}_{\bar \alpha}(A(\alpha), \s(\alpha))
    \text{ implies } a \in A\fm.
  \end{equation}
  Then property  \eqref{property-Am1} is preserved under quotients by
  $\Supp \alpha$ and, if $\Supp \alpha=\{0\}$, under  
  localization at $\fm$. More
  precisely:
  \begin{enumerate}
    \item Let $R_1 := R/\Supp \alpha$. Then
    property \eqref{property-Am1} holds for $(A_1, \s_1) := 
    (A \ox_R R_1,\s\ox_R\id_{R_1})$ and the ordering $\alpha_1$ induced by
      $\alpha$ on $R_1$ together with 
      the maximal ideal $\fm_1 := \fm/\Supp \alpha$
      of $R_1$, i.e.,
      \begin{equation*}
        \forall a \in \Sym(A_1,\s_1) \quad a \ox_{R_1} 1_{\kappa(\alpha_1)} \in M^{\CM_{\bar\alpha_1}}_{\bar
        \alpha_1}(A_1(\alpha_1),\s_1(\alpha_1)) \text{ implies } a \in A_1\fm_1.
      \end{equation*}
      Furthermore, if $\alpha\in \Sperm R$, then $\alpha_1 \in \Sperm R_1$. 
    \item Assume that $\Supp \alpha = \{0\}$. Then
      property \eqref{property-Am1} holds for $(A_2, \s_2) := 
      (A \ox_R R_\fm,\s\ox_R \id_{R_\fm})$ and 
      the ordering $\alpha_2$ induced by $\alpha$ on $R_\fm$ together
      with the unique maximal ideal $\fm_2$ of $R_\fm$, i.e.,
      \begin{equation*}
        \forall a \in \Sym(A_2,\s_2) \quad a \ox_{R_\fm} 1_{\kappa(\alpha_2)}
         \in M^{\CM_{\bar\alpha_2}}_{\bar
        \alpha_2}(A_2(\alpha_2),\s_2(\alpha_2)) \text{ implies } a \in A_2\fm_2.
      \end{equation*}
      Furthermore, if $\alpha\in \Sperm R$, then $\alpha_2 \in \Sperm R_\fm$. 
  \end{enumerate}
\end{lemma}
\begin{proof}
(1) We have natural maps
      \[R \rightarrow R_1 = R/\Supp \alpha \rightarrow \Qf(R_1) =
      \kappa(\alpha) = \kappa(\alpha_1)\]
      with $\bar \alpha = \bar \alpha_1$,
      and thus
      \[A \rightarrow A_1 = A \ox _R R_1 \rightarrow  
      A \ox_R \Qf(R_1)=A(\alpha) = A_1(\alpha_1),\]
      while
      \[
        \s_1(\alpha_1) := \s_1 \ox_{R_1} \id_{\kappa(\alpha_1)} 
          = (\s \ox_R \id_{R_1}) \ox_{R_1} \id_{\kappa(\alpha_1)}
          = \s \ox_R \id_{\kappa(\alpha_1)}
          = \s(\alpha).
      \]

      Let $b \in \Sym(A_1,\s_1)$ be such that 
      \[
      b \ox_{R_1}
      1_{\kappa(\alpha_1)} \in
      M^{\CM_{\bar\alpha_1}}_{\bar \alpha_1}(A_1(\alpha_1), \s_1(\alpha_1)) = 
      M^{\CM_{\bar\alpha}}_{\bar
      \alpha}(A(\alpha), \s(\alpha)).
      \]
      Then $b = c \ox_R 1_{R_1}$ for
      some $c \in A$  (the argument is similar to Remark~\ref{M-ox})  and
      \[b \ox_{R_1} 1_{\kappa(\alpha_1)} = c \ox_R 1_{R_1} \ox_{R_1}
      1_{\kappa(\alpha_1)} = c \ox_R 1_{\kappa(\alpha)}.\]
      Therefore $c \ox_R
      1_{\kappa(\alpha)} \in M^{\CM_{\bar\alpha}}_{\bar \alpha}(A(\alpha), \s(\alpha))$
       and thus $c \in A\fm$ by \eqref{property-Am1}. It follows that $b = c
      \ox_R 1_{R_1} \in A\fm \ox_R 1_{R_1} \subseteq (A \ox_R R_1) (\fm / \Supp 
      \alpha)= A_1 \fm_1$
      (the inclusion follows from $am \ox 1 = (a \ox 1) (m \ox 1) = (a \ox
      1) (1 \ox (m + \Supp \alpha))$).
      
      The statement about the maximality of $\alpha_1$ follows from the fact
      that the homeomorphism in \cite[Proposition~3.3.11]{KSU} clearly
      preserves inclusions.

(2) Note that $\alpha_2$ is indeed an ordering on $R_\fm$ since $(R
      \setminus \fm) \cap \Supp \alpha = \emptyset$, cf. 
      \cite[Proposition~3.3.10]{KSU}. We also have
      \[\alpha_2 = \Big\{\dfrac{r}{s^2} \,\Big|\, r \in \alpha,\ s \in R 
      \setminus
      \fm\Big\}\]
      (by \cite[Proof of Proposition~3.3.10]{KSU})
      and $\Supp \alpha_2 = \{0\}$, so that $\kappa(\alpha_2) = \Qf(R_\fm)$.
      Observe that the map $R \rightarrow \Qf(R)$, which is the first step in
      the computation of signatures of elements of $\Sym(A,\s)$
       (since $\Supp
      \alpha = \{0\}$, cf. Definition~\ref{def:sig_h}) 
      factors through $R_\fm$, giving  
      $R \rightarrow R_\fm \rightarrow
      \Qf(R)$, with $\Qf(R) = \Qf(R_\fm)$, i.e., $\kappa(\alpha) =
      \kappa(\alpha_2)$. Finally, a direct verification shows that $\bar \alpha =
      \bar \alpha_2$.

      Let $b \in \Sym(A_2,\s_2)$ be such that $b \ox_{R_\fm}
      1_{\kappa(\alpha_2)} \in M^{\CM_{\bar\alpha_2}}_{\bar \alpha_2}(A_2(\alpha_2), \s_2(\alpha_2))$. Then
      $b = c \ox_R \dfrac{1}{s}$ for some $s \in R\setminus \fm$ 
      (the argument is
      again similar to Remark~\ref{M-ox}). 
      Since $s$ is
      invertible in $R_\fm$, $b$ has the same signature at $\alpha_2$ as $bs^2 =
      cs \ox_R 1_{R_\fm}$, so that $(cs \ox_R 1_{R_\fm}) \ox_{R_\fm}
      1_{\kappa(\alpha_2)} \in M^{\CM_{\bar\alpha_2}}_{\bar \alpha_2}(A_2(\alpha_2), \s_2(\alpha_2))$.

      We have $ (cs \ox_R 1_{R_\fm})
      \ox_{R_\fm} 1_{\kappa(\alpha_2)}=cs \ox_R 1_{\kappa(\alpha)}$, and thus
      \begin{align*}
        cs \ox_R 1_{\kappa(\alpha)} \in & M^{\CM_{\bar\alpha_2}}_{\bar \alpha_2}(A_2(\alpha_2),
          \s_2(\alpha_2)) \\
        & = M^{\CM_{\bar\alpha_2}}_{\bar \alpha_2}((A \ox_R R_\fm) \ox_{R_\fm}
          \Qf(R_\fm), (\s \ox \id_{R_\fm}) \ox \id_{\kappa(\alpha_2)}) \\
        &= M^{\CM_{\bar\alpha}}_{\bar \alpha}(A \ox_R \kappa(\alpha), \s \ox \id_{\kappa(\alpha)}).
      \end{align*}
      By property \eqref{property-Am1}, we obtain $cs \in A\fm$. Therefore $cs
      \ox_R 1_{R_\fm} \in A\fm \ox_R 1_{R_\fm} \subseteq (A \ox_R R_\fm) \fm_2$
      (the inclusion follows from $am \ox_R 1_{R_\fm} = (a \ox_R 1_{R_\fm})(m
      \ox_R 1_{R_\fm}) = (a \ox_R 1_{R_\fm}) (1 \ox_R m)$) and thus 
      $b=c \ox_R
      \dfrac{1}{s} = (cs \ox_R 1_{R_\fm}) \dfrac{1}{s^2} \in 
      (A \ox_R R_\fm) \fm_2$.
      
      The statement about the maximality of $\alpha_2$ follows from the fact
      that the homeomorphism in \cite[Proposition~3.3.10]{KSU} clearly
      preserves inclusions.
\end{proof}

\begin{lemma}\label{rescue}
  Let $\fm$ be a maximal ideal of $R$ and let
  $\alpha \in \Sperm R$.  Then
  there is $a \in \Sym(A,\s)$ such that $a \ox 1 \in M^{\CM_{\bar\alpha}}_{\bar \alpha}(A(\alpha),
  \s(\alpha))$ and $a \not \in A\fm$.
\end{lemma}
\begin{proof}
  We assume that the conclusion does not hold, so that 
  property~\eqref{property-Am1} of Lemma~\ref{improvement} holds.
  We proceed in four steps:
  \medskip

  \emph{Step~1:} Take $a \in \Sym(A,\s)$ such that $a
      \ox 1 \in M^{\CM_{\bar\alpha}}_{\bar \alpha}(A(\alpha), \s(\alpha))$, cf. 
      Lemma~\ref{sl-case-part1}.
      Then $a \in A\fm$ by property~\eqref{property-Am1}.
      Furthermore, for every $r \in \Supp \alpha$ we have
      $\sign^\CM_\alpha \qf{a+r}_\s=\sign^\CM_\alpha \qf{a}_\s$, cf.
      Definition~\ref{def:sig_h} (since the first step in the computation
      is scalar extension to $\Qf(R/\Supp \alpha)$). Thus, by hypothesis,  
      $a + r \in A\fm$. Since $a
      \in A\fm$, we get that $\Supp \alpha \subseteq A\fm \cap R = \fm$
      (cf. \cite[Corollary~7.1.2(1)]{ford17} for the equality).
\medskip

  \emph{Step~2:}  We first apply Lemma~\ref{improvement}(1) and get that we can 
  assume
      that $\Supp \alpha = \{0\}$, and in particular that $R$ is a domain. It is
      then possible to apply Lemma~\ref{improvement}(2) and we can also assume
      that $R$ is a local domain with maximal ideal $\fm$.
 \medskip
  
    \emph{Step~3:}  Since $R$ is a local domain and $\Supp \alpha = \{0\}$, the 
    following holds:
      
      For every $\frac{r_1}{s_1} \in \Qf(R)$ with $r_1, s_1 \in \alpha
      \setminus \{0\}$, there exists $r \in \alpha\cap R^\x$ such that $r \le_\alpha
      \frac{r_1}{s_1}$.

      Proof of this claim: By the description of $\alpha$ 
      in Theorem~\ref{alpha-semilocal}, there is $r_1' \in \alpha\cap R^\x$ such
      that $r_1' \le_\alpha r_1$, so that $\frac{r_1'}{s_1} \le_\alpha
       \frac{r_1}{s_1}$.
      Observe that if there is $s'_1 \in \alpha\cap R^\x$ such that $s'_1
      \ge_\alpha s_1$, then $\frac{r_1'}{s'_1} \le_\alpha
      \frac{r_1}{s_1}$, and we
      can take $r := r'_1 {s'_1}^{-1}$.

      However, since $R$ is local, such an $s'_1$ exists: If $s_1$ is
      invertible, we take $s'_1 = s_1$. If $s_1$ is not invertible, then $s_1
      \in \fm$. Therefore $1+s_1
      \not \in \fm$, i.e., $1+s_1 \in R^\x$, and of course $1+s_1 \in \alpha$.
      We then take $s'_1 := 1+s_1$. This proves the claim.
\medskip

  \emph{Step~4:} We work in the central simple algebra with involution 
      $(A(\alpha),\s(\alpha)) = (A \ox_R \Qf(R) ,\s\ox\id)$, 
      and denote by $\bar \alpha$ the ordering induced by 
      $\alpha$
      on $\Qf(R)$.  By Lemma~\ref{with-cofinal}(2), 
      there is $\frac{r_1}{s_1} \in
      \bar \alpha \setminus \{0\}$ such that 
      $\qf{a \ox 1 - \dfrac{r_2}{s_2}}_{\s(\alpha)}$ has
      maximal signature at $\bar \alpha$ for every $\frac{r_2}{s_2} \in \bar
      \alpha$ such that $\frac{r_2}{s_2} \le_{\bar \alpha} \frac{r_1}{s_1}$.
      In other words, $a \ox 1 - \frac{r_1}{r_2} \in M^{\CM_{\bar\alpha}}_{\bar \alpha}(A(\alpha),
      \s(\alpha))$.
      
      By Step~3, 
      there is $r \in \alpha\cap R^\x$ such that $r \le_{\bar
      \alpha} \frac{r_1}{s_1}$. In particular we have
      $(a -r) \ox 1 \in M^{\CM_{\bar\alpha}}_{\bar
      \alpha}(A(\alpha), \s(\alpha))$. Therefore, $a-r \in A\fm$ 
      by property~\eqref{property-Am1}
      and thus $r \in
      A\fm$. But this is impossible since $A\fm$ is a proper ideal and $r$ is
      invertible.
\end{proof}

\begin{lemma}\label{not-included}
  Assume that $R$ is semilocal, let $\fm$ be a maximal ideal of $R$, and let
  $\alpha \in \Sperm R$.  Then there
  is $b_\fm \in \CCS_\alpha(A,\s)$ such that $b_\fm + A\fm \in (A/A\fm)^\x$.
\end{lemma}
\begin{proof}
  Let $a \in \Sym(A,\s)$ be such that $a \ox 1 \in M^{\CM_{\bar\alpha}}_{\bar
  \alpha}(A(\alpha),\s(\alpha))$ and $\pi_{A\fm}(a)\not=0$,
  cf. Lemma~\ref{rescue}.
  Then $D_{(A/A\fm, \bar \s)} (k\x \qf{\pi_{A\fm}(a)}_{\bar \s})
  \subseteq \CCS'_\fm$ for all $k\in\N$ by definition of $\CCS'_\fm$.
  Since $(A/A\fm, \bar \s)$ is a central simple algebra with
  involution, \cite[Lemma~2.4]{A-U-pos} applies and there is $\ell \in \N$ such
  that $\ell \x \qf{\pi_{A\fm}(a)}_{\bar \s}$ represents an invertible
  element $b'$.
  Since $b' \in \CCS'_\fm$ and $\pi_{A\fm}$ is surjective, we have
  \[b' = \sum_{j=1}^\ell \bar \s(\pi_{A\fm}(x_j)) \pi_{A\fm}(a) \pi_{A\fm}(x_j),\]
  with  $x_j \in A$. Therefore, we take $b_\fm =
  \sum_{j=1}^\ell \s(x_j) a x_j$.
\end{proof}

\begin{prop}\label{sl-max-sign}
  Assume that $R$ is semilocal and let $\alpha\in \Sperm R$.  Then there 
  are invertible elements in $\CCS_\alpha(A,\s)$. Furthermore, every invertible 
  element $a\in \CCS_\alpha(A,\s)$
  satisfies  $\sign^\CM_\alpha \qf{a}_\s = m_{\bar
  \alpha}(A(\alpha), \s(\alpha))$. 
\end{prop}
\begin{proof}
  Let $\fm_1, \ldots, \fm_\ell$ be the maximal ideals of $R$.
  Observe that for each $i \in \{1, \ldots, \ell\}$ there is $b_i \in
  \CCS_\alpha(A,\s)$ such that $b_i + A\fm_i \in (A/A\fm_i)^\x$ by
  Lemma~\ref{not-included}.
  By the Chinese remainder theorem, the canonical map
    $\xi \colon  R \rightarrow R/\fm_1 \x \cdots \x R/\fm_{\ell}$
  is surjective. In particular there are $r_1, \ldots, r_{\ell} \in R$ such
  that
  $\xi(r_i) = (0, \ldots, 0,1,0, \ldots, 0)$,
  where the coordinate $1$ is the one corresponding to the quotient $R/{\fm_i}$.
  Define
  \[b:= \s(r_1)b_1r_1 + \cdots + \s(r_{\ell})b_{\ell} r_{\ell}.\]
  Observe that $b \in \CCS_\alpha(A,\s)$.
  We check that $b$ is invertible. By Proposition~\ref{first-ideals}, it
  suffices to show that $b+A \fm$ is invertible in $A/A\fm$ for every maximal
  ideal $\fm$ of $R$. Consider such an ideal $\fm_i$. By definition of $r_i$ we
  have $r_i + \fm_i = 1+ \fm_i$ and $r_j + \fm_i = 0 + \fm_i$ for all $j \not
  = i$. Therefore $b + A\fm_i = b_i + A\fm_i$, which is invertible in
  $A/A\fm_i$.

  \medskip
  
  We show that if $a\in\CCS_\alpha(A,\s)$ is invertible, then
  $\sign^\CM_\alpha \qf{a}_\s = m_{\bar \alpha}(A(\alpha),
  \s(\alpha))$. Since $a \in \CCS_\alpha(A,\s)$, there are $a_1, \ldots, a_k \in
  \Sym(A,\s)$ such that $a \in D_{(A,\s)} \qf{a_1,\ldots,a_k}_\s$ and
  $
  a_1 \ox 1, \ldots, a_k \ox 1 \in M^{\CM_{\bar\alpha}}_{\bar \alpha}(A(\alpha), \s(\alpha))$.
  Since $a$ is
  invertible, a standard argument gives 
  $\qf{a_1, \ldots, a_k}_\s\simeq \qf{a}_\s \perp h $ 
  for some hermitian form $h$ over $(A,\s)$.
  Extending the scalars to $\kappa(\alpha)$, we obtain $\qf{a \ox
  1}_{\s(\alpha)} \perp h \ox \kappa(\alpha) \simeq \qf{a_1 \ox 1, \ldots, a_k
  \ox 1}_{\s(\alpha)}$ over the central simple algebra with involution
  $(A(\alpha), \s(\alpha))$. Observe that $a_1 \ox 1, \ldots, a_k \ox 1$ are
  invertible in $A(\alpha)$, and thus that $h \ox \kappa(\alpha)$ is
  nonsingular. By \cite[Lemma~2.2]{A-U-pos} there is $\ell \in \N$
  such that $\ell \x (h\ox \kappa(\alpha)) \simeq \qf{c_1, \ldots, c_t}_{\s(\alpha)}$ for some $c_1,
  \ldots, c_t \in \Sym(A(\alpha)^\x,\s(\alpha))$ (they are invertible since $h
  \ox \kappa(\alpha)$ is nonsingular). Therefore,
  \[\ell \x \qf{a \ox 1}_{\s(\alpha)} \perp \qf{c_1, \ldots, c_t}_{\s(\alpha)}
  \simeq \ell \x \qf{a_1 \ox 1, \ldots, a_k \ox 1}_{\s(\alpha)}.\]
  Note that both forms are diagonal over $(A(\alpha), \s(\alpha))$, so
  that $\ell + t =  \ell k$ for dimension reasons. Since $a_1 \ox 1, \ldots, a_k
  \ox 1 \in M^{\CM_{\bar\alpha}}_{\bar \alpha}(A(\alpha), \s(\alpha))$, the form on the right-hand
  side has the maximal signature that can be obtained by a 
  nonsingular diagonal form 
  of dimension $\ell k$
  over
  $(A(\alpha), \s(\alpha))$, 
  namely $\ell k \cdot m_{\bar\alpha}(A(\alpha), \s(\alpha))$. It is therefore the same for the form on the
  left-hand side, which implies that $a\ox 1 $  (and every $c_i$) 
  belongs to
  $M^{\CM_{\bar\alpha}}_{\bar \alpha}(A(\alpha),
  \s(\alpha))$, i.e., $\sign^\CM_\alpha \qf{a}_\s =
  m_{\bar \alpha}(A(\alpha), \s(\alpha))$.
\end{proof}

Since $m_\alpha(A,\s) \leq m_{\bar \alpha}(A(\alpha), \s(\alpha))$ by the
definition of signatures, 
the following corollary is an immediate consequence of
Proposition~\ref{sl-max-sign}:

\begin{cor}\label{same-max}
  Assume that $R$ is semilocal and that $\alpha\in \Sperm R$. Then
  \begin{enumerate}
    \item $m_\alpha(A,\s) = m_{\bar \alpha}(A(\alpha), \s(\alpha))$.
    \item If $a \in M^\CM_\alpha(A,\s)$, then $a \ox 1 \in 
    M^{\CM_{\bar \alpha}}_{\bar \alpha}(A(\alpha), \s(\alpha))$.
  \end{enumerate}
  Note that if 
    $\alpha\not\in \Nil[A,\s]$, then
    $m_{\bar \alpha}(A(\alpha), \s(\alpha))=n_{\bar \alpha}$.
\end{cor}

\begin{proof}
  The final statement is the only one that still requires a proof and 
  follows from \cite[Proposition~6.7]{A-U-pos}
  since $\bar\alpha\not\in \Nil[A(\alpha), \s(\alpha)]$ by 
  Lemma~\ref{Nil-product}.  
\end{proof}

As already mentioned in Remark~\ref{sign-bounded}, we have
$\sign^\CM_\alpha \qf{a}_\s \le  n_{\bar \alpha}$ for $a\in \Sym(A,\s)$
by 
\cite[Proposition~4.4]{A-U-PS}. It
immediately follows from Corollary~\ref{same-max} that the definition of
$m_\alpha(A,\s)$ could include non-invertible elements when $R$ is semilocal:

\begin{cor}
  Assume that $R$ is semilocal and that $\alpha \in \Sperm R$. Then
  \[m_\alpha(A,\s) = \max \{\sign^\CM_\alpha \qf{a}_\s \mid a \in \Sym(A,\s)\}.\]
\end{cor}

\begin{rem}
  While an element of
  maximal signature
  in a central simple algebra with involution over a field  is necessarily 
  invertible (cf. Lemma~\ref{with-cofinal}), this may not be so for  Azumaya
  algebras with involution, even already in the ring case. 
  For example,
  let $(A,\s) =
  (\Z_{3\Z}, \id)$ (in particular, hermitian forms over $(A,\s)$ are just
  bilinear forms over $\Z_{3\Z}$ and their signatures are the usual Sylvester
  signatures). The ring $\Z_{3\Z}$ has a unique ordering $\alpha_0$, and
  $\sign_{\alpha_0} \qf{3} = 1 = m_{\alpha_0}(A,\s)$.
\end{rem}

\section{The involution trace pairing}\label{sec:inv-tr-p}
Let $R$ be a commutative ring (with $2\in R^\x$), let
$(A,\s)$ be an Azumaya algebra with involution over $R$, let $S=Z(A)$
and let $\iota:=\s|_S$. 
Note that $A$ is Azumaya over $S$, but not necessarily Azumaya
over $R$.

\subsection{The involution trace form}

We consider the reduced trace of $A$, $\Trd_A\colon A\to S$, cf.
\cite[IV, \S 2]{KO}, and recall that it is additive and $S$-linear.
Furthermore, $\Trd_A$ commutes with scalar extensions of $S$  since 
its computation does not depend on the choice of splitting ring, cf. 
\cite[IV, Proposition~2.1]{KO}. In fact, $\Trd_A$ also commutes with scalar 
extensions of $R$ (the case of interest to us) as the following computation shows:

\begin{lemma}
  Let $R'$ be a commutative ring that contains $R$. Then for all $a\in A$,
  \[
    \Trd_A(a) \ox_R 1_{R'} = \Trd_{A \ox_R R'}(a \ox_R 1_{R'}).
  \]
\end{lemma}

\begin{proof}
Observe that $A \ox_R R' \cong A \ox_S (S \ox_R R')$, via 
$a \ox_R r' \mapsto a \ox_S 1
\ox_R r'$. It follows that
\begin{align*}
  \Trd_A(a) \ox_R 1_{R'} &= \Trd_A(a) \ox_S 1_{S} \ox_R  1_{R'} \\
    &= \Trd_{A \ox_S S \ox_R R'}(a \ox_S 1_S \ox_R 1_{R'})  \\
    &= \Trd_{A \ox_R R'}(a \ox_R 1_{R'}).\qedhere
  \end{align*}
\end{proof}

\begin{lemma}\label{lem-tr}
  For all $a\in A$ we have $\Trd_A(\s(a)) = \iota(\Trd_A(a))$.
\end{lemma}

\begin{proof}
If  $\iota=\Id_S$, i.e., $S=R$, the statement follows from \cite[III,
  (8.1.1)]{knus91}. We assume that $\iota \not = \Id_S$, and first observe that
  \begin{align*}
    \Trd_A&(\s(a)) = \iota(\Trd_A(a)) \Leftrightarrow \forall \fp \in \Spec R \ 
        \Trd_A(\s(a)) \ox_R 1_{R_\fp} = \iota(\Trd_A(a)) \ox_R 1_{R_\fp} \\
      &\Leftrightarrow \forall \fp \in \Spec R \ 
        \Trd_{A \ox_R R_\fp}(\s(a) \ox 1_{R_\fp}) = (\iota \ox \id)(\Trd_A(a)
        \ox 1_{R_\fp}) \\
      &\Leftrightarrow \forall \fp \in \Spec R \ 
        \Trd_{A \ox_R R_\fp}((\s \ox \id)(a \ox 1_{R_\fp})) = 
       (\iota \ox
        \id)(\Trd_{A \ox_R R_\fp}(a \ox 1_{R_\fp})).
  \end{align*}
  Therefore, it suffices to prove the result for $(A \ox_R R_\fp, \s \ox
  \id)$, and in particular it suffices to prove the statement of the lemma under
  the extra hypothesis that $R$ is local.
  In this case, the arguments in \cite[(2.15) and
  (2.16)]{BOI} hold \emph{mutatis mutandis} for the Azumaya algebra with
  involution $(A,\s)$ over $R$. More precisely, since $R$ is local,
  there is $\lambda \in S$ such that $\lambda^2 \in R^\x$,
  $\iota(\lambda) = -\lambda$ and $S = R \oplus \lambda R$, cf. 
  Section~\ref{sec:quadalg}. Then
  \cite[(2.15)]{BOI} holds for $(A,\s)$, i.e., the map
  \[
    (A\ox_R S, \s\ox\id) \to (A\x A^\op, \ve),\ a\ox s \mapsto
    (as, (\s(a)s)^\op),
  \] 
  where $\ve$ denotes the exchange involution,
  is an isomorphism of $S$-algebras with involution
  (replacing the element $\alpha$ in the proof of \cite[(2.15)]{BOI}
  by $\lambda$, and observing that $\lambda - \iota(\lambda) = 2\lambda
  \in S^\x$). The claimed equality becomes
  straightforward to verify after application of this isomorphism since the
  reduced trace is invariant under scalar extension.
\end{proof}

The \emph{involution trace form}  of $(A,\s)$
is the form
\[
  T_\s\colon  A\x A\to S,\ (x,y)\mapsto \Trd_A(\s(x)y).
\]
By Lemma~\ref{lem-tr}, $T_\s$ is symmetric bilinear 
over $R$ if $S=R$ and hermitian over 
$(S,\iota)$ otherwise.

\begin{lemma}
  The form $T_\s$ is nonsingular, i.e., $(A, T_\s)\in \bH(S,\iota)$.
\end{lemma}

\begin{proof}
  Since $S$ is a finitely generated $R$-module (cf. Definition~\ref{first-def}
  and Proposition~\ref{prop:Az}), the form $T_\s$ is nonsingular
  if and only if the form $T_\s \ox_R R/\fm$ is nonsingular for all maximal
  ideals $\fm$ of $R$, cf. \cite[I, Lemma~7.1.3]{knus91}. Since $\Trd_A$
  commutes with scalar extension, $T_\s \ox_R R/\fm$ is isometric to the 
  involution trace form of the central simple algebra with involution
  $(A \ox_R R/\fm, \s \ox \id)$, which is nonsingular, cf. \cite[\S 11]{BOI}.
\end{proof}

\begin{lemma}\label{lem:isom}
  The sandwich map $\sw$ (cf. \eqref{swm})
  induces an isomorphism
    \begin{equation*}
      (A\ox_S  A^\op, \s\ox \s^\op) \cong (\End_S(A), \ad_{T_\s})
    \end{equation*}
  of Azumaya algebras with involution over $R$.
\end{lemma}

\begin{proof}
  By \eqref{swm}, the sandwich map is an isomorphism of $S$-algebras.
  We first show that it respects the involutions, i.e., 
  $\ad_{T_\s}(\sw(a\ox b^\op))= \sw(\s(a)\ox
  \s^\op(b^\op))$ for
  all $a,b\in A$. With reference to \eqref{ad} this follows from the
  straightforward computation
  \begin{align*}
    T_\s(x, \sw(\s(a)\ox \s^\op(b^\op))(y)) &= \Trd_A(\s(x)\s(a)y\s(b))
    =\Trd_A(\s(b)\s(x)\s(a)y)\\
    &=\Trd_A(\s(axb)y)
    =T_\s (axb,y)\\
    &=T_\s(\sw(a\ox b^\op)(x),y)
  \end{align*}
  which holds for all $x,y,a,b\in A$. 
    
  Finally, $(\End_S(A), \ad_{T_\s})$ is an Azumaya algebra with 
  involution over $R$. 
  Indeed, by Proposition~\ref{prop:Az}, $S$ is finite \'etale over $R$
  and $A$ is Azumaya over $S$.
  Hence, $\End_S(A)$ is Azumaya over $S$ and in particular has centre $S$. 
  Therefore, by Proposition~\ref{prop:Az} $\End_S(A)$ is projective
  and separable over $R$. Clearly, $\ad_{T_\s}$ is $R$-linear. Thus we
  just have to show that if $s\in S\setminus R$, then 
  $\ad_{T_\s}(s\cdot \id_A)\not=s\cdot \id_A$. This can be checked 
  by showing that if $\ad_{T_\s}(s\cdot \id_A)=s\cdot \id_A$ for some $s\in S$,
  then $\s(s)=s$,  
  using the
  nonsingularity of $T_\s$ and the  
  definition of $\ad_{T_\s}$ in a similar fashion to the computation 
  above. 
\end{proof}

\subsection{The Goldman element}\label{sec:gm}

We can view $\Trd_A$ as an element of $\End_S(A)$.
By the definition of the sandwich isomorphism \eqref{swm}, there is a unique element
$g_A=\sum_i x_i\ox y^\op_i$ in $A\ox_S A^\op$ such that
\begin{equation}\label{G-eq2}
  \sw(g_A)(a)=\sum_i x_i a y_i = \Trd_A(a)\quad\text{for all } a\in A.
\end{equation}
The element $g_A$ is called the \emph{Goldman 
element} of $A$. 

\begin{lemma}\label{golds}
  The Goldman element $g_A$ satisfies
  \begin{equation}\label{G-eq7}
    (\s\ox \s^\op)(g_A)=g_A.
  \end{equation}  
\end{lemma}

\begin{proof}
  By Lemma~\ref{lem:isom} it  suffices to show that
  $\ad_{T_\s}(\Trd_A)=\Trd_A$ in $\End_S(A)$
   in order to prove the claim. 
  Consider \eqref{ad} with $h=T_\s$ and
  $f=\Trd_A$. Using the properties of $T_\s$ and $\Trd_A$ we have
  \begin{align*}
    T_\s(x, \ad_{T_\s}(\Trd_A)(y))&=T_\s(\Trd_A(x),y)
    =\iota(\Trd_A(x)) T_\s(1,y)\\
    &=\Trd_A(\s(x)) T_\s(1,y)    
    = T_\s(x,1)\Trd_A(y) \\
    &= T_\s(x, \Trd_A(y)) 
    \end{align*}
  for all $x,y\in A$. Since $T_\s$ is nonsingular, the claim follows.
\end{proof}

We usually think of $g_A$ as an element of $A\ox_S A$ via the 
canonical $S$-module isomorphism $A\ox_S A \to A\ox_S A^\op,
a\ox b \mapsto a\ox b^\op$ and write $g_A=\sum_i x_i\ox y_i$, cf. 
\cite[p.~112]{KO}.
Since  $\s\ox\s$ and $\s\ox\s^\op$ correspond to each other as additive maps
under this isomorphism, Lemma~\ref{golds} yields
\begin{equation}\label{soxs}
      (\s\ox \s)(g_A)=g_A.
\end{equation}
in $A\ox_S A$.
Furthermore, we have
\begin{equation}\label{G-eq3}
  g_A^2=1
\end{equation}
and
\begin{equation}\label{G-eq4}
  g_A(a\ox b)=(b\ox a)g_A\quad \text{for all } a,b\in A,
\end{equation}
cf. \cite[IV, Proposition~4.1]{KO}.

\subsection{Module actions}\label{module-actions}
We define $\io A$ to be the  $S$-algebra given by the ring $A$ equipped with
the following left action by $S$:
\[S \x A \rightarrow A, \quad (s,a) \mapsto s \iop a := a\iota(s).\]
We denote this action of $S$ on $A$ with the symbol $\iop{}$
in order to distinguish it from the product in $A$ of elements of $S$ and $A$.
Note that $(\io A,\s)$ is an $(S,\iota)$-algebra with 
involution 
as 
presented in Section~\ref{hfri}, and that
if  $(S, \iota) = (R, \id)$, then $\io A
= A$.

We can view $A$ as a left $A\ox_S \io A$-module via the \emph{twisted} 
sandwich action:
\begin{equation}\label{twsa}
  a\ox  b \tsa x := ax\s(b)
\end{equation}
for all $a,b,x \in A$ (it is necessary to use $\io A$ in the tensor product
instead of $A$, in order for the action to be well-defined, which is the
motivation for introducing $\io A$).

\begin{lemma}\label{opio-isom}
 The twisted sandwich map 
 \[
 A\ox_S \io A \to \End_S(A), a\ox  b \mapsto [x\mapsto ax\s(b)]
 \] 
 induces an isomorphism
  \[
    (A\ox_S \io A, \s\ox \s)\cong (\End_S(A), \ad_{T_\s})
  \]
  of Azumaya algebras with involution over $R$.
\end{lemma}

\begin{proof}
  The map $\s\colon A^\op \to \io A$ is an isomorphism of $S$-algebras and yields
  an isomorphism $(A\ox_S A^\op, \s\ox\s^\op)\cong (A\ox_S \io A, \s\ox\s)$ of 
  Azumaya algebras with involution over $R$. The result then follows from 
  Lemma~\ref{lem:isom}.
\end{proof}

For a right $A$-module $M$ we denote by $\io M$ the right $\io A$-module
with the same elements as $M$, with multiplication $\io M\x \io A\to \io M$
the same as the multiplication $M\x A \to M$ and with left action by $S$ given
by
\[S \x \io M \rightarrow \io M, \quad (s,m) \mapsto s \iop m := m\iota(s).\]
If $(M,h)\in \Herm^{\ve}(A,\s)$, then $h$ is $\iota(\ve)$-hermitian on $\io
M$, and we  denote it by $\io h$.  

If $M_1$ and $M_2$ are right
$A$-modules a direct verification shows that $M_1 \ox_S \io M_2$ is a right $A
\ox_S \io A$-module with multiplication  induced by
\[ 
(M_1 \ox_S \io M_2)\x (A \ox_S \io A) \rightarrow M_1 \ox_S \io M_2, \quad
(m_1 \ox m_2) \cdot (a \ox b) := m_1a \ox m_2b.
\]

\subsection{The involution trace pairing}\label{sec:trp}

Let $(M_1,h_1)\in \Herm^{\ve_1}(A,\s)$ and $(M_2,h_2)\in \Herm^{\ve_2}(A,\s)$,
and consider the involution trace form $(A,T_\s)\in \bH(S,\iota)$.
Using the twisted sandwich action of $A\ox_S \io A$ on $A$ (cf. \eqref{twsa}),
we can define the $S$-module $(M_1\ox_S \io M_2) \ox_{A\ox_S \io A} A$, which
carries the form $T_\s \knp (h_1\ox \io h_2) \in 
\Herm^{\ve_1\iota(\ve_2)}(S,\iota)$, where $\knp$
denotes the product of forms from \cite[I, (8.1), (8.2)]{knus91}. 
In other words,
\[
  T_\s \knp (h_1\ox \io h_2) (m_1\ox  m_2\ox a, m_1'\ox  m_2' \ox a') :=
  T_\s\bigl(a,h_1(m_1,m_1') \ox \io h_2( m_2,  m_2') \tsa a'  \bigr)
\]
for all $m_1, m_1'\in M_1$, $m_2,m_2'\in \io M_2$ and $a,a'\in A$.
(Note that we
do not indicate all parentheses in long tensor products of elements, in order
not to overload the notation.)

For this
product to be well-defined,  $A$ needs to be  an 
$(A\ox_S \io
A)$-$S$ bimodule   and the form $(A,T_\s)$ needs to ``admit'' $A\ox_S
\io A$, which means that the equality
\[
  T_\s(\s(x)\ox \s(y)\tsa a, b) =T_\s(a, x\ox  y\tsa b)
\]
must hold for all $x,y,a,b\in A$, but this follows from Lemma~\ref{opio-isom}.

In this way we obtain the pairing
\begin{equation}\label{eq:ns}
  \begin{split}
  *\colon  &\Herm^{\ve_1}(A,\s)\x \Herm^{\ve_2}(A,\s) \to \Herm^{\ve_1 
  \iota(\ve_2)}(S,\iota),\\
(M_1,h_1) &* (M_2, h_2) := \bigl((M_1\ox_S  \io M_2) \ox_{A\ox_S  \io A} A,
T_\s \knp  (h_1\ox \io h_2)\bigr).
\end{split}
\end{equation}
Expanding the definition of the pairing and simply writing $h_1 *h_2$, we see 
that
\begin{align}
    h_1 * h_2 (m_1\ox  m_2\ox a,  m_1'\ox  m_2' \ox a') 
    &= T_\s\bigl(a,h_1(m_1,m_1') \ox  \io h_2( m_2, m_2') 
    \tsa a'  \bigr)\notag\\
    &=\Trd_A\bigl(\s(a)h_1(m_1, m_1')a'\s(h_2(m_2,m_2'))\bigr)\label{G-eq8}\\
    &=\Trd_A\bigl(h_1(m_1a, m_1'a') \s( h_2(m_2,m_2'))\bigr).\notag
\end{align}

\begin{lem}\label{ts-morita}
  The  Azumaya algebras with involution $(A\ox_S \io A, \s\ox\s)$ and  
  $(S,\iota)$ are Morita equivalent via
  \[
    \Herm^\ve (A\ox_S \io A, \s\ox\s)\to \Herm^\ve (S,\iota),\
    \vf \mapsto T_\s \knp \vf.
  \]
\end{lem}

\begin{proof}
  This follows from Lemma~\ref{opio-isom} and \cite[I, 
  Theorem~9.3.5]{knus91}.
\end{proof}

\begin{cor}\label{lem:preserves} The following properties hold:
  \begin{enumerate}
    \item The pairing $*$ preserves
    orthogonal sums in each component.
    \item If $h_1$ and $h_2$ are nonsingular, then $h_1* h_2$ is nonsingular.
    \item If $h_1\simeq h_1'$ and $h_2\simeq h_2'$, then $h_1 * h_2 \simeq
    h_1' * h_2'$.
  \end{enumerate}
\end{cor}

\begin{proof}
  Let $h_i\in \Herm^{\ve_i}(A,\s)$ for $i=1,2$, then
  $h_1 *h_2= T_\s \knp (h_1\ox \io h_2)$ and the three statements follow from
  Lemma~\ref{ts-morita},  \cite[I, Theorem~9.3.5]{knus91}  and standard
  properties of the tensor product of forms.
\end{proof}

\begin{thm}\label{pairing-ass-2}
  Let $(M_i,h_i) \in \Herm^{\ve_i}(A,\s)$ for $i=1,2,3$.  Then
  \[(h_1 * h_2) \ox_S h_3 \simeq (h_3 * h_2) \ox_S h_1.\]
\end{thm}

\begin{proof}
  We are grateful to the first referee for suggesting this proof, which
  is significantly shorter and more conceptual than our original one.
  
  Consider the $\ve_1\iota(\ve_2)\ve_3$-hermitian forms $(M,h):=(M_1\ox_S \io M_2\ox_S M_3, h_1\ox \io h_2 \ox h_3)$
  and $(M',h'):=(M_3\ox_S \io M_2\ox_S M_1, h_3\ox \io h_2 \ox h_1)$
  over $(A\ox_S \io A \ox_S A, \s^{\ox 3})$. Let 
  $g'_A \in A\ox_S \io A \ox_S A$ be the image of the Goldman element
  $g_A \in A\ox_S A$ under the natural map $a\ox b \mapsto a\ox 1\ox b$.
  Using the properties of $g'_A$, induced by those of $g_A$ 
  (cf. Section~\ref{sec:gm}), the following
  computation shows that $(M,h)$ and 
  $(M',h')$ are isometric via the  isomorphism of right
  $A\ox_S \io A \ox_S A$-modules
  $M\to M'$, defined by
  $m_1\ox m_2\ox m_3 \mapsto (m_3\ox m_2\ox m_1)g'_A$:  
  \begin{align*}
    h'( (m_3\ox m_2 \ox m_1)g'_A, & (m'_3\ox m'_2\ox m'_1)g'_A) \\
    &=\s^{\ox 3}(g'_A) h'(m_3\ox m_2 \ox m_1, m'_3\ox m'_2\ox m'_1)g'_A\\
    &=g'_A(h_3(m_3,m'_3) \ox \io h_2(m_2,m'_2) \ox h_1(m_1, m'_1) )g'_A\\
    &=h_1(m_1, m'_1) \ox \io h_2(m_2,m'_2) \ox h_3(m_3,m'_3) \\
    &= h(m_1\ox m_2 \ox m_3, m'_1\ox m'_2\ox m'_3).
  \end{align*}
  
  Since $(A,\s)\cong (\End_A(A), \ad_{\qf{1}_\s})$, it follows from
  Lemmas~\ref{opio-isom} and \ref{adj-prod} that
  \begin{align*}
      (A\ox_S \io A \ox_S A, \s^{\ox 3}) &\cong (\End_S(A)\ox_S A, 
      \ad_{T_\s}\ox \s)\\
      &\cong (\End_S(A)\ox_S \End_A(A), 
      \ad_{T_\s}\ox \ad_{\qf{1}_\s})\\
      &\cong (\End_{S\ox_S A}(A\ox_S A), \ad_{T_\s \ox \qf{1}_\s}),
  \end{align*}
  which yields the Morita equivalence
  \begin{equation}\label{knp-mo}
    \Herm^\ve (A\ox_S \io A \ox_S A, \s^{\ox 3})\to 
    \Herm^\ve (S\ox_S A,\iota\ox \s),\
    \vf \mapsto (T_\s\ox \qf{1}_\s) \knp \vf
  \end{equation}
  by \cite[I, Theorem~9.3.5]{knus91}. 
  The isomorphism of right $A \ox_S A$-modules
  \[((M_1 \ox_S \io M_2) \ox_S M_3) \ox_{(A \ox_S \io A) \ox_S A} (A \ox_S A)
  \rightarrow ((M_1 \ox_s \io M_2) \ox_{A \ox_S \io A} A) \ox_S (M_3 \ox_A A),\]
  \[m_1 \ox m_2 \ox m_3 \ox a \ox b \mapsto m_1 \ox m_2 \ox a \ox m_3 \ox b,\]
  followed by  the isomorphism of right $A$-modules $M_3 \ox_A A \rightarrow
  M_3$, $m_3 \ox b \mapsto m_3b$, then yield the isometries
  \[
    (T_\s \ox \qf{1}_\s)  \knp (h_1\ox \io h_2 \ox h_3)\simeq
    T_\s  \knp (h_1\ox \io h_2)  \ox \qf{1}_\s  \knp h_3 \simeq
    (h_1 * h_2) \ox h_3. 
  \]
  A similar argument shows that  $(T_\s \ox 
  \qf{1}_\s)  \knp (h_3\ox \io h_2 \ox h_1) \simeq (h_3 * h_2) \ox h_1$. 
  The result then follows from the isometry $(M,h)\simeq (M',h')$ since
  \eqref{knp-mo} preserves isometries. 
\end{proof}

\begin{rem}
  The pairing $*$ was introduced and studied in detail
  for $\ve$-hermitian forms over
  central simple algebras with involution by Garrel in
  \cite{garrel-2023}. (A similar construction for
  quaternion algebras had already been considered by Lewis \cite{Lew}, using
  the norm form instead of the involution trace form of the quaternion
  conjugation, cf. \cite[Remark~4.4]{garrel-2023}.) 
  In our presentation we stayed close to Garrel's approach via hermitian
  Morita theory. We are grateful to
  the second referee for suggesting an alternative
  approach via the $S$-linear isomorphism
  \[
    M_1 \ox_A \iss M_2 \to (M_1\ox_S \io M_2)\ox_{A\ox_S \io A} A,\
    m_1\ox m_2 \mapsto (m_1\ox m_2)\ox 1
   \]  
   (where $\iss M_2$ is the left $A$-module obtained by twisting the right
   $A$-module structure of $M_2$ by $\s$)
   with inverse
   \[
     (M_1\ox_S \io M_2)\ox_{A\ox_S \io A} A \to  M_1 \ox_A \iss M_2,
     (m_1\ox m_2)\ox a \mapsto m_1 a \ox m_2 = m_1 \ox m_2 \s(a),
   \]
   from which the pairing $*$ can be defined directly as
   \[
     h_1 * h_2 (m_1 \ox m_2, m_1' \ox m_2') := 
     \Trd_A\bigl(h_1(m_1,m_1'), \s(h_2(m_2, m_2'))   \bigr).
   \]
\end{rem}

We finish this section with a number of results for later use. We first consider
\cite[Proposition~4.9]{garrel-2023} in our context:

\begin{lemma}\label{exb*c}
  If $b,c\in \Sym(A,\s)$, then 
  $\qf{b}_\s * \qf{c}_\s\simeq \vf_{b,c}$, where $(A,\vf_{b,c})\in 
  \bH(S,\iota)$ is given by
  \[
    \vf_{b,c}\colon A\x A \to S,\ (x,y)\mapsto \Trd_A(\s(x)by c).
  \]
\end{lemma}

\begin{proof}
  The form $\qf{b}_\s * \qf{c}_\s$ is defined on the $S$-module
  $(A\ox_S \io A)\ox_{A\ox_S \io A} A$.  
  Since $A$ is a left $A \ox_S \io A$-module, the left action of $A \ox_S \io A$
  induces an isomorphism of left $A \ox_S \io A$-modules (and thus of $S$-modules):
    \[
      f\colon  (A\ox_S \io A)\ox_{A\ox_S \io A} A \to A,\ x\ox y\ox a \mapsto 
      (x\ox y) \tsa a=
      xa\s(y).
    \]
    Using \eqref{G-eq8}, we verify that it is the required isometry:
    \begin{align*}
      \vf_{b,c} (&f(x_1\ox y_1\ox a_1),f(x_2\ox y_2\ox a_2) )\\ 
        &= \vf_{b,c}(x_1a_1\s(y_1), x_2a_2\s(y_2))
        =\Trd_A(y_1\s(a_1)\s(x_1)bx_2a_2\s(y_2) c)\\
        &=\Trd_A(\s(a_1)\s(x_1)bx_2a_2\s(y_2) c y_1)
        =\Trd_A(\s(x_1a_1)bx_2a_2 \s(\s(y_1) c y_2))\\
        &=\Trd_A(\qf{b}_\s(x_1a_1,x_2a_2) \s(\qf{c}_\s (y_1,y_2)))\\
        &=\qf{b}_\s * \qf{c}_\s (x_1\ox y_1\ox a_1, x_2\ox y_2\ox a_2 ).\qedhere
    \end{align*}
\end{proof}

Next we show that $*$ is well-behaved under scalar extensions, and start with 
the
following lemma (for which we could not find a reference):

\begin{lemma}\label{horrible-tensor}
  Let $\Lambda$ be a commutative ring, and let 
  $B$ and $C$ be $\Lambda$-algebras. Let $M_1$ be a right $B$-module and 
  $M_2$  a left
  $B$-module. Then
  \begin{align*}
    (M_1 \ox_B M_2) \ox_\Lambda 
    C &\rightarrow (M_1 \ox_\Lambda C) \ox_{B \ox_\Lambda C} (M_2 \ox_\Lambda
    C), \\
    (m_1 \ox m_2) \ox c &\mapsto (m_1 \ox 1) \ox (m_2 \ox c)
  \end{align*}
  is an isomorphism of right $C$-modules.
\end{lemma}

\begin{proof}
  Let $f$ be the map defined in the statement of the lemma and let
  \begin{align*}
    g\colon  (M_1 \ox_\Lambda C) \ox_{B \ox_\Lambda C} (M_2 \ox_\Lambda C)
     &\rightarrow (M_1 \ox_B M_2) \ox_\Lambda C, \\
    (m_1 \ox c_1) \ox (m_2 \ox c_2) &\mapsto (m_1 \ox m_2) \ox c_1c_2.
  \end{align*}
  A standard (but lengthy) verification shows that $f$ and $g$ are well-defined
  and additive, are inverses of each other, and that $f$ is right $C$-linear.
\end{proof}

\begin{lemma}\label{*-ext-scalars}
  Let $T$ be a commutative $R$-algebra. Then
  \[(h_1 * h_2) \ox_R T \simeq (h_1 \ox_R T) * (h_2 \ox_R T).\]
\end{lemma}

\begin{proof}
  The form $(h_1 * h_2) \ox_R T$ is defined on 
  $((M_1 \ox_S \io M_2) \ox_{A \ox_S \io A} A) \ox_R T$.
  Using Lemma~\ref{horrible-tensor} twice we have 
  \begin{align}
    ((M_1 \ox_S & \io M_2) \ox_{A \ox_S \io A} A) \ox_R T \notag \\ 
    & \cong ((M_1 \ox_S \io M_2) \ox_R T)
    \ox_{(A \ox_S \io A) \ox_R T} (A \ox_R T)  \notag \\ 
    &\cong ((M_1 \ox_R T) \ox_{S \ox_R T} (\io M_2 \ox_R T)) 
    \ox_{(A \ox_R T) \ox_{S
    \ox_R T} (\io A \ox_R T)} (A \ox_R T),   \label{*-2} 
  \end{align}
  and
  the successive isomorphisms are given by
  \[ [(m_1 \ox m_2) \ox a] \ox t \mapsto [(m_1 \ox m_2) \ox 1] \ox (a \ox t)
  \mapsto [(m_1 \ox 1) \ox (m_2 \ox 1)] \ox (a \ox t).\]
  We denote the composition of these two isomorphisms by $\xi$.

  By definition, $(h_1 * h_2) \ox_R T = (T_\s \knp (h_1 \ox_S \io h_2)) \ox_R T$,
  while 
  \[
  (h_1 \ox_R T) * (h_2 \ox_R T) = T_{\s \ox \id_T} \knp ((h_1 \ox_R T)
  \ox_{S\ox_R T} (\io h_2 \ox_R T)) 
  \]
  and is defined on the module \eqref{*-2} (using that $Z(A
  \ox_R T) = Z(A) \ox_R T$, cf. Proposition~\ref{change-of-base}). We check that
  $\xi$ is an isometry:
  \begin{align*}
    &T_{\s \ox \id_T} \knp ((h_1 \ox_R T) \ox_{S \ox_R T} (\io h_2 \ox_R T))
    (\xi(m_1 \ox m_2 \ox
    a \ox t), \xi(m'_1 \ox m'_2 \ox a' \ox t') ) \\
    &= \Trd_{A \ox T} \left((h_1 \ox T)((m_1 \ox 1)(a \ox t), 
    (m'_1 \ox 1)(a' \ox t'))\right.\\ 
    &\qquad\qquad\qquad\qquad\qquad\qquad\qquad 
    \left. \cdot (\s \ox \id_T)((h_2 \ox T)(m_2 
    \ox 1, m'_2 \ox 1)) \right) \\
    &= \Trd_{A \ox T} \left((h_1 \ox T)(m_1a \ox t, m'_1a' \ox t') (\s \ox
    \id_T)((h_2 \ox T)(m_2 \ox 1, m'_2 \ox 1))\right) \\
    &= \Trd_{A \ox T} \left((h_1(m_1a,m'_1a') \ox tt') (\s \ox
    \id_T)(h_2(m_2, m'_2) \ox 1)\right) \\
    &= \Trd_{A \ox T}(h_1(m_1a,m'_1a') \s(h_2(m_2, m'_2)) \ox tt') \\
    &= \Trd_{A}(h_1(m_1a, m'_1a') \s(h_2(m_2, m'_2))) \ox tt' \\
    &= ((T_\s \knp (h_1 \ox \io h_2)) \ox_R T)((m_1 \ox m_2 \ox a) \ox t, (m'_1 \ox m'_2
    \ox a') \ox t'), 
  \end{align*}
  where we used that $\Trd_A$ commutes with scalar extension in the penultimate
  step.
\end{proof}

\begin{rem}
  Recall that when $S$ is not connected, the Witt groups $W^\ve(A,\s)$
  and $W^\ve(S,\iota)$ are trivial, cf. Remark~\ref{rem:notcon}.
  Therefore the construction of the pairing $*$ is not interesting in this case.
\end{rem}

\section{Pairings and PSD quadratic forms}\label{link-with-psd}

Throughout this section we assume that the commutative ring $R$ 
(with $2\in R^\x$) is semilocal.
Let $(A,\s)$ be an Azumaya algebra with involution over $R$,  $S=Z(A)$
and $\iota=\s|_S$.

Let $(M,h)$ be a hermitian form over $(S,\iota)$. Then 
$h(x,x)\in \Sym(S,\iota)=R$ for all $x\in M$. Therefore we say that $h$ is
positive semidefinite (resp. negative  semidefinite)
at $\alpha \in \Sper R$ if $h(x,x)\in \alpha$ (resp. $h(x,x)\in -\alpha$ )   
for all $x\in M$. We use the standard abbreviations PSD and NSD.

\begin{lemma}\label{qf-PSD-no-pc}
  Let $(M,h)$ be a hermitian form over $(S,\iota)$ and let $\alpha \in \Sper
  R$. Then $h$ is PSD at $\alpha$ if and only if $h \ox \kappa(\alpha)$ is PSD
  at the ordering $\bar \alpha$ induced by $\alpha$ on $\kappa(\alpha)$.
\end{lemma}
\begin{proof}
  We  use the description of $M
  \ox_R \kappa(\alpha)$ given in Remark~\ref{M-ox}.

  ``$\Rightarrow$'': Let $m \ox \frac{1}{\bar b} \in M \ox_R \kappa(\alpha)$,
  where $m \in M$ and $b \in R \setminus \Supp (\alpha)$. Then
  \[h \ox \kappa(\alpha)(m \ox \dfrac{1}{\bar b}, m \ox \dfrac{1}{\bar b}) = 
  \overline{h(m,m)} (\dfrac{1}{\bar b})^2,\]
  which is in $\bar \alpha$ by hypothesis.

  ``$\Leftarrow$'': Let $m \in M$. We have
  $h \ox \kappa(\alpha)(m \ox 1, m \ox 1) = \overline{h(m,m)}$,
  which belongs to $\bar \alpha$, and thus $h(m,m) \in \alpha$.
\end{proof}

\begin{rem}\label{PSD-NSD}
  Let $F$ be a real closed field and let $P$ be the unique ordering on $F$.
  Recall from Remark~\ref{rem:nil} and the beginning of Section~\ref{sec3}
  that if $(A,\s) = (M_n(D),\vartheta^t)$ with $(D, \vartheta) \in 
  \{(F, \id), (F(\sqrt{-1}), \gamma), ((-1,-1)_F, \gamma)\}$, where $\gamma$
  denotes conjugation, resp. quaternion conjugation,
  and $a \in \Sym(A,\s)$, then $\sign^{\CM_P}_P \qf{a}_\s$ is plus or minus the standard
  Sylvester signature of $a$ at $P$.
  
  Indeed, if $a$ is a matrix in $\Sym(M_n(D), \vartheta^t)$ such that
  $\vartheta^t(a)=a$, then 
  under the hermitian Morita equivalence between $\Herm^\ve (M_n(D), \vt^t)$
  and $\Herm^\ve(D,\vt)$ in \cite[equation (2.1)]{A-U-pos}, 
  the one-dimensional hermitian form $\qf{a}_{\vt^t}$ over
  $(M_n(D), \vartheta^t)$ corresponds to  
  the hermitian form over $(D,\vartheta)$
  whose Gram matrix is $a$, and the signature of $\qf{a}_{\vt^t}$ is 
  defined to be the 
  signature of the matrix $a$. The definition of signature of hermitian forms
  allows the use of a different Morita equivalence $\CM_P$
   than this particular one, which may
  result in a change of sign.

  In particular, there exists $\ve\in\{-1,1\}$ (which depends only on $\CM_P$)
  such that for all 
  $a\in \Sym(M_n(D), \vartheta^t)$, if $\sign^{\CM_P}_P \qf{a}_{\vt^t}$ is 
  maximal (among signatures of hermitian forms of the form $\qf{b}_{\vt^t}$),
  then $\ve a$ is a PSD matrix. 
\end{rem}

\begin{lemma}\label{trace-csa-no-pc}
  Let $F$ be a field and $(B,\tau)$ a central simple $F$-algebra with 
  involution.  
  Let $P \in X_F\setminus \Nil[B,\tau]$.
  There is $\delta \in \{-1,1\}$ such that for every $b,c \in 
  M^{\CM_P}_P(B,\tau)$
  and every $x \in B$, $\Trd_B(\tau(x)bx c ) \in \delta \cdot P$.
\end{lemma}

\begin{proof}
  By \cite[Proposition~6.7 and equation~(6.1)]{A-U-pos}, we have
  \[m_P(B,\tau) = m_{P'}(B \ox_F F_P, \tau \ox \id),\]
  where $P'$ denotes the unique ordering on $F_P$. Hence,
  $x \in M^{\CM_P}_P(B,\tau)$ implies $x \ox 1 \in M^{\CM_{P'}}_{P'}(B
  \ox_F F_P, \tau \ox \id)$.
  Therefore, we can assume that $F$ is real closed with unique ordering $P$,
  that 
  $(B,\tau) = (M_n(D), \Int(a)\circ \vt^t)$ for some $a \in M_n(D)^\x$
  with $\vt$ of the same kind as
  $\tau$ and
  $(D,\vt)  \in \{(F, \id), (F(\sqrt{-1}), \gamma),
  ((-1,-1)_F, \gamma)\}$,  where $\gamma$
  denotes conjugation, resp. quaternion conjugation 
  (cf. \cite[First page of Section~2.3]{A-U-PS}).
  Note that $\vt(a)^t = \delta a$ for some $\delta \in \{-1,1\}$
  by \cite[Propositions~2.7 and 2.18]{BOI}.
  Furthermore, the algebra $(B,\tau)$ is formally real in the sense of
  \cite[Definition~3.2]{A-U-pos} since $P \not \in \Nil[B,\tau]$ (apply
  \cite[Proposition~6.6]{A-U-pos}, where $\widetilde X_F := X_F \setminus
  \Nil[B,\tau]$), and it follows that we can choose the involution $\vt$
  such that $\delta = 1$ by \cite[Corollary~3.8]{A-U-pos}.

  With reference to \cite[Remark~3.13]{A-U-Kneb}, observe that
  \[  a^{-1}\qf{b}_\tau =\qf{a^{-1}b}_{\Int(a^{-1})\circ \tau}
  =\qf{a^{-1}b}_{\vt^t},\]
  and that there is a Morita equivalence $\CM'_P$ such that
  \[
    m_P(B,\tau)=  \sign_P^{\CM_P} \qf{b}_\tau =  
    \sign_P^{\CM'_P} a^{-1}\qf{b}_\tau =  \sign_P^{\CM'_P} \qf{a^{-1}b}_{\vt^t}.
  \]
  Similarly we have $\sign^{\CM'_P}_P \qf{a^{-1}c}_{\vt^t} = m_P(B,\tau)$.
  By Remark~\ref{PSD-NSD} there exists $\varepsilon \in \{-1,1\}$ 
  such that $\ve a^{-1}b$ and $\ve a^{-1}c$ are both PSD with respect to $P$.
  Write $b':=\ve a^{-1} b $ 
  and $c':= \ve a^{-1} c $. 
  The matrix  $\vt(a)^t c' a$  is then PSD over $(D,\vt)$, and thus is a hermitian
  square in $(M_n(D), \vt^t)$, i.e.,  there
  exists a matrix $c''\in M_n(D)$ such that $\vt(a)^t c' a = \vt(c'')^t c''$
  (this is a classical consequence of the principal axis theorem, which also
  holds for quaternion matrices by  \cite[Corollary~6.2]{Zhang}, cf. 
  \cite[Appendix~A]{A-U-pos}).  
  It follows that
  \begin{align*}
    \Trd_B(\tau(x)bxc) &= \Trd_B((\Int(a)\circ \vt^t)(x)(a \varepsilon b')x(a 
    \varepsilon c')) 
                     = \Trd_B(a \vt(x)^t a^{-1}a b' x a c') \\
                     &= \Trd_B(a \vt(x)^t b'xac') 
                     = \delta \Trd_B(\vt(x)^t b'x \vt(a)^t c' a) \\
                     &= \delta \Trd_B(\vt(x)^t b'x \vt(c'')^t c'' )
                     = \delta \Trd_B(c''\vt(x)^t b'x \vt(c'')^t),
  \end{align*}
  which belongs to $\delta P$ since $b'$ is PSD.
\end{proof}

\begin{prop}\label{quadratic0-no-pc}
  Let $\alpha \in \Sper R \setminus \Nil[A,\s]$.
  Then 
  there is $\delta \in \{-1,1\}$ such that for
  every $b, c \in M^\CM_\alpha(A,\s)$, the hermitian form $\delta (\qf{b}_\s *
  \qf{c}_\s)$ is nonsingular and PSD with respect to $\alpha$.
\end{prop}

\begin{proof}
  The forms $\qf{b}_\s$ and $\qf{c}_\s$ are nonsingular since $b$ and $c$ are
  invertible, and thus $\qf{b}_\s * \qf{c}_\s$ is nonsingular by 
  Corollary~\ref{lem:preserves}.
  By Lemma~\ref{qf-PSD-no-pc}, it suffices to show that the form $(\delta
  \qf{b}_\s * \qf{c}_\s) \ox \kappa(\alpha)$ is PSD with respect to $\bar \alpha$.
  Using now that $(\qf{b}_\s * \qf{c}_\s) \ox \kappa(\alpha) \simeq \qf{b \ox 1}_{\s \ox
  \id} * \qf{c \ox 1}_{\s \ox \id}$ by Lemma~\ref{*-ext-scalars}, 
  and that $b \ox 1, c
  \ox 1 \in M^{\CM_{\bar\alpha}}_{\bar \alpha}(A(\alpha), \s(\alpha))$ by
  Corollary~\ref{same-max}, we
  can assume that $(A,\s)$ is a central simple algebra with involution 
  over the field $\kappa(\alpha)$ with ordering
  $\bar\alpha$ that is non-nil by Lemma~\ref{Nil-product}.  
  The result then follows from
  Lemma~\ref{trace-csa-no-pc} since $\qf{b}_\s * \qf{c}_\s$ is isometric to
  the hermitian 
  form  $\vf_{b,c}$, cf. Lemma~\ref{exb*c}.
\end{proof}

\begin{lemma}\label{quadratic-no-pc}
  Let $\alpha \in \Sper R \setminus \Nil[A,\s]$, 
  and let $b, c \in M^\CM_\alpha(A,\s)$. Then  there
  are nonsingular hermitian forms  $\varphi_1$ and 
  $\varphi_2$ over $(S,\iota)$
  that are PSD at $\alpha$ such that $\varphi_1 \ox_S \qf{b}_\s 
  \simeq  \varphi_2 \ox_{S} \qf{c}_\s$.
\end{lemma}
\begin{proof}
  Let $\delta \in \{-1,1\}$ be as given by Proposition~\ref{quadratic0-no-pc}.
   Then the hermitian forms
  $\delta (\qf{c}_\s * \qf{c}_\s)$ and $\delta (\qf{b}_\s * \qf{c}_\s)$ over 
  $(S, \iota)$ are PSD with respect to $\alpha$. 
  Furthermore, they are nonsingular since $b$ and $c$ are invertible and by
  Corollary~\ref{lem:preserves}.
  By Theorem~\ref{pairing-ass-2} we have 
  \[(\qf{c}_\s * \qf{c}_\s) \ox_{S} \qf{b}_\s \simeq (\qf{b}_\s * \qf{c}_\s)
  \ox_{S} \qf{c}_\s,\]
  hence
  \[\delta (\qf{c}_\s * \qf{c}_\s) \ox_S \qf{b}_\s \simeq \delta(\qf{b}_\s 
  * \qf{c}_\s) \ox_S \qf{c}_\s,\]
  and we conclude with Proposition~\ref{quadratic0-no-pc}.
\end{proof}

\section{Sylvester's law of inertia and Pfister's local-global principle}

Let $(A,\s)$ be an Azumaya algebra with involution over $R$,  $S=Z(A)$
and $\iota=\s|_S$.
In the case of central simple algebras with involution over fields, a weaker
version of the following result appears in \cite[Theorem~8.9]{A-U-pos}.

\begin{thm}[Sylvester's law of inertia]\label{Sylvester-no-pc}
  Assume that $R$ is semilocal connected and let 
  $\alpha \in \Sper R \setminus \Nil[A,\s]$. 
  Then $S$ is connected and, letting $t := \rk_S A$, we have:
  \begin{enumerate}
    \item Let $h$ be a nonsingular hermitian form over $(A,\s)$.  For every $c
      \in M^\CM_\alpha(A,\s)$ there are $w_1, \ldots, w_t, u_1, \ldots, u_r, v_1,
      \ldots, v_s \in \alpha \cap R^\x$ such that
      \[\qf{w_1, \ldots, w_t} \ox_R h \simeq \qf{u_1, \ldots, u_r} \ox_R
      \qf{c}_\s \perp \qf{-v_1, \ldots, -v_s} \ox_R \qf{c}_\s.\]
  
    \item Assume
    $      \qf{a_1,\ldots,a_r}_\s \perp \qf{-b_1,\ldots, -b_s}_\s 
      \simeq \qf{a'_1,\ldots,a'_p}_\s \perp \qf{-b'_1,\ldots, -b'_q}_\s$
  with $a_1,\ldots, a_r$, $b_1,\ldots, b_s, a'_1,\ldots, a'_p, 
  b'_1,\ldots, b'_q\in M^\CM_\alpha(A,\s)$. Then $r=p$ and $s=q$.
  \end{enumerate}
\end{thm}

\begin{proof}
    We first observe that $S$ is connected by Remark~\ref{S-connected}. By
    Lemma~\ref{S-diag}  every nonsingular hermitian form over
    $(S,\iota)$ is diagonalizable. By Proposition~\ref{prop:Az}, $A$ is Azumaya
    over $S$, hence a projective $S$-module and so $\rk_S A$ is defined, and is
    constant since $S$ is connected, cf. \cite[p.~12]{Saltman99}.

  (1)
  By Theorem~\ref{pairing-ass-2} we have 
  $(\qf{c}_\s * \qf{c}_\s) \ox_S h \simeq
   (h * \qf{c}_\s) \ox_S \qf{c}_\s$.  Since
  $h * \qf{c}_\s$ is nonsingular by Corollary~\ref{lem:preserves} 
  and diagonalizable, we can write
  \[h * \qf{c}_\s \simeq \qf{u_1, \ldots, u_r}_\iota \perp \qf{-v_1, \ldots,
  -v_s}_\iota,\]
  with $u_1, \ldots, u_r, v_1, \ldots, v_s \in \alpha \cap R^\x$. Therefore
  \begin{align*}
    (\qf{c}_\s * \qf{c}_\s) \ox_S h &\simeq (h * \qf{c}_\s) \ox_S
        \qf{c}_\s \\
      &\simeq \qf{u_1, \ldots, u_r}_\iota \ox_S \qf{c}_\s \perp \qf{-v_1,
        \ldots, -v_s}_\iota \ox_S \qf{c}_\s.
  \end{align*}

  By Proposition~\ref{quadratic0-no-pc}, 
  the nonsingular hermitian form $\qf{c}_\s * \qf{c}_\s$ over
  $(S,\iota)$ is PSD or NSD with respect to $\alpha$. Up to replacing it by
  its opposite, we can assume it is PSD with respect to $\alpha$. As observed
  above, it is diagonalizable, 
  and by Lemma~\ref{exb*c} it is defined on $A$ and therefore has 
  dimension $\rk_S A = t$.
  Hence there are $w_1, \ldots, w_t \in \alpha \cap
  R^\x$ such that $\qf{c}_\s * \qf{c}_\s \simeq 
  \qf{w_1, \ldots, w_t}_\iota$, and
  thus
  \[\qf{w_1, \ldots, w_t}_\iota \ox_S h \simeq \qf{u_1, \ldots, u_r}_\iota
  \ox_S \qf{c}_\s \perp \qf{-v_1, \ldots, -v_s}_\iota \ox_S \qf{c}_\s.\]
  The result now follows from Lemma~\ref{R-Z-isom}.
  
  (2) For dimension reasons (after localization, since $A$ is a projective 
  $R$-module)
  we have $r+s = p+q$. The result will follow if we show
  $r-s = p-q$, i.e., $r+q = p+s$.  We have
  the following equality in the Witt group $W(A,\s)$:
  \[\qf{a_1,\ldots,a_r,b'_1,\ldots,b'_q}_\s = \qf{a'_1,\ldots, a'_p, b_1,
  \ldots, b_s}_\s,\]
  which implies that
  \[\qf{a_1,\ldots,a_r,b'_1,\ldots,b'_q}_\s \perp \vf  \simeq 
  \qf{a'_1,\ldots, a'_p, b_1,
  \ldots, b_s}_\s \perp \psi, \]
  where $\vf$ and $\psi$ are hyperbolic forms over $(A,\s)$. Taking signatures
  on both sides yields $(r+q)m_\alpha(A,\s) = (p+s)m_\alpha(A,\s)$
  since hyperbolic forms have signature zero. The result follows. 
\end{proof}

For $r_1,\ldots, r_\ell  \in R$, we use the notation 
$\Pf{r_1, \ldots, r_{\ell}}$ to denote the
Pfister form $\qf{1,r_1} \ox_R \cdots \ox_R \qf{1,r_\ell}$.

\begin{cor}\label{prod-hyperbolic-no-pc2}
  Assume that $R$ is semilocal connected. 

  \begin{enumerate}
    \item Assume that $\s$ is of orthogonal or symplectic type (so that
      $(S,\iota)=(R,\id)$).  
    Let $\alpha \in \Nil[A,\s]$
      and let $h$ be a nonsingular hermitian form over $(A,\s)$. Then there is a
      nonsingular quadratic form $q$ over $R$ of dimension $\rk_R A$ 
      that is PSD at
      $\alpha$ and such that $q \ox_R h$ is hyperbolic.

    \item Let $\alpha \in \Sper R \setminus \Nil[A,\s]$,  let $a,b \in
      M^\CM_\alpha(A,\s)$, and let $t:=\rk_S A$. Then there are 
      $\ell \in \N$ and $r_1, \ldots, r_{\ell},
      w_1, \ldots, w_t \in \alpha \cap R^\x$ such that 
      \[ (\qf{w_1, \ldots, w_t} \ox_R \Pf{r_1, \ldots, r_{\ell}}) \ox_R  \qf{a,-b}_\s\]
      is hyperbolic. 
  \end{enumerate}
\end{cor}
\begin{proof}
  (1)
  By Lemma~\ref{skew-inv-no-pc} there exists a skew-symmetric element 
  $a\in A^\x$. Let $\tau=\Int(a)\circ \s$ and note that 
  $\tau$ is orthogonal if $\s$ is symplectic and vice versa (indeed,
  by Proposition~\ref{rem:quad-et} it suffices to check it for the
  central simple $\kappa(\alpha)$-algebra with involution 
  $(A(\alpha), \s(\alpha))$, where it is true by \cite[Proposition~2.7]{BOI}). 
  It follows from Remark~\ref{rem:nil} and
  Lemma~\ref{Nil-product} (1)$\Leftrightarrow$(2) 
  that $\alpha\in \Sper(R) \setminus \Nil[A,\tau]$. Scaling $h$ by $a$
  gives the nonsingular skew-hermitian form $h':=ah$ over $(A,\tau)$. 
  
  Let $c \in M^\CM_\alpha(A,\tau)$.  Then the form $\qf{c}_\tau$ is 
  nonsingular, and
  so the pairing $\qf{c}_\tau * \qf{c}_\tau$ is also nonsingular by
  Corollary~\ref{lem:preserves}.  By Proposition~\ref{quadratic0-no-pc}, there
  exists $\delta\in\{-1,1\}$ such that the form $\delta (\qf{c}_\tau *
  \qf{c}_\tau)$  is PSD at $\alpha$.  
  
  We now consider the form $\delta(\qf{c}_\tau * \qf{c}_\tau) \ox_R h'$ which is
  isometric to $\delta (h' * \qf{c}_\tau) \ox_R \qf{c}_\tau$ by
  Theorem~\ref{pairing-ass-2}.  But $h' * \qf{c}_\tau$ is a pairing of a
  skew-hermitian and a hermitian form over $(A,\tau)$, and so is skew-symmetric
  over $(S,\iota)=(R,\id)$ by \eqref{eq:ns}.   
  Hence, $h' * \qf{c}_\tau$ is
  hyperbolic by \cite[I, Corollary~4.1.2]{knus91} since every projective
  $R$-module is free by \cite{hinohara}. 
  
  Therefore, letting $q:= \delta (\qf{c}_\tau * \qf{c}_\tau)$, the form $q\ox_R
  h'$ is hyperbolic and since scaling by $a^{-1}$ commutes with tensoring by
  $q$, we obtain that $q\ox_R h$ is hyperbolic by \cite[I,
  Theorem~9.3.5]{knus91} applied to the scaling-by-$a^{-1}$ Morita equivalence.
  Note that $q$ is diagonalizable since $R$ is semilocal and $A$ is projective
  over $R$, and that the dimension of $q$ is $\rk_R A$, cf. Lemma~\ref{exb*c}.

  (2) Observe that $\rk_S A $ is constant by the first paragraph of the
  proof of  Theorem~\ref{Sylvester-no-pc}.
  Let $c \in M^\CM_\alpha(A,\s)$. By Theorem~\ref{Sylvester-no-pc}, there are 
  $w_1, \ldots, w_t, u_1, \ldots, u_r, v_1, \ldots, v_s \in \alpha
  \cap R^\x$ such that
  \[\qf{w_1, \ldots, w_t} \ox_R \qf{a,-b}_\s \simeq \qf{u_1, \ldots, u_r} \ox_R
  \qf{c}_\s \perp \qf{-v_1, \ldots, -v_s} \ox_R \qf{c}_\s.\]
  The signature at $\alpha$ of the left-hand side is $0$, so we must have
  $r=s$. Therefore,
  \begin{align*}
    \qf{w_1, \ldots, w_t} \ox_R \qf{a,-b}_\s &\simeq \qf{u_1, \ldots, u_r} \ox_R \qf{c}_\s \perp
      \qf{-v_1, \ldots, -v_r} \ox_R \qf{c}_\s \\
    &\simeq \qf{\bar u} \ox_R \qf{c}_\s \perp \qf{-\bar v} \ox_R \qf{c}_\s 
    \simeq \qf{\bar u, -\bar v} \ox_R \qf{c}_\s,
  \end{align*}
  where we write $\bar u$ for $u_1, \ldots, u_r$ and similarly for 
  $\bar v$. We use the notation $\Pf{\bar u, \bar v}$ for
  the Pfister form $\Pf{u_1, \ldots, u_r, v_1, \ldots, v_s}$. 
  The form $\Pf{\bar u, \bar v} \ox_R \qf{\bar u, -\bar v}$ is hyperbolic since
  \begin{align*}
    \Pf{\bar u, \bar v} \ox_R \qf{\bar u, -\bar v} &\simeq 
    \Pf{\bar u, \bar v} \ox_R \qf{\bar
      u} \perp \Pf{\bar u, \bar v} \ox_R \qf{-\bar v} \\
      &\simeq  \bigperp_{i=1}^r \underbrace{\Pf{\bar u, \bar v} \ox_R 
      \qf{u_i}}_{\Pf{\bar u, \bar v}} \perp \bigperp_{i=1}^r -
  \underbrace{\Pf{\bar u, \bar v} \ox_R \qf{v_i}}_{\Pf{\bar u, \bar v}} \\
      &\simeq r \x \Pf{\bar u, \bar v} \perp -r \x \Pf{\bar u, \bar v},
  \end{align*}
  where the final equality holds since Pfister forms over $R$ are round (cf.
  \cite[Corollary~2.16]{baeza} which uses the fact that the hypothesis $2 \in
  R^\x$ ensures that $|R/\mathfrak{m}| > 2$ 
  for every maximal ideal $\mathfrak{m}$ of $R$). Therefore
  $(\qf{w_1, \ldots, w_t} \ox_R \Pf{\bar u, \bar v}) \ox_R \qf{a, -b}_\s$ is
  hyperbolic, proving the result. 
\end{proof}

The main idea of the
proof of the next result comes from Marshall's proof of Pfister's
local-global principle in \cite[Theorem~4.12]{marshall80}.

\begin{prop}\label{PLG-V1-no-pc}
  Let $R$ be semilocal connected and let $t_0 := \rk_R A$.  Let $h$ be a
  nonsingular hermitian form over $(A,\s)$ such that for every $n \in \N$,
  the form $2^n t_0^2 \x h$ is not
  hyperbolic. Then there is $\alpha \in \Sperm R$ such that
  $\sign^\CM_\alpha h \not = 0$.
\end{prop}
\begin{proof}
  Observe that if $t=\rk_S A$ is defined, then $t_0$ is equal to $t$ (if $R=S$)
  or $2t$ (if $R \not = S$). We identify nonsingular
  hermitian forms with their classes in
  the Witt group $W(A,\s)$, considered as a $W(R)$-module. Note that a 
  nonsingular hermitian
  form is hyperbolic if and only if its class is zero since we assume that $2\in
  A^\x$ (and therefore metabolic forms are hyperbolic).

  For the convenience of the reader we first present 
  the ideas of the
  three main steps of the proof before giving the full details:
  
  \begin{enumerate} 
    \item We define a maximal non-empty set $\CS$ of (nonsingular) Pfister
    forms over $R$ such that $p \cdot h \not = 0$ for every $p \in \CS$, and
    such that $\CS$ is closed under products. This final property produces an
    ideal $J := \bigcup_{p \in \CS} \Ann_{W(A,\s)}(p)$, and $h \not \in J$ by
    construction (this is linked to the notion of Pfister quotient in
    \cite[Chapter~4.7]{marshall80}).

    The actual construction in the proof below is slighly different (a factor
    $t_0^2$ appears for technical reasons), and the ideal is called $J_\CS$.

    \item The maximality of $\CS$ ensures that the set $\alpha_0$ of all
    elements of $R^\x$ represented by these Pfister forms is ``almost'' an
    ordering, more precisely: $\alpha_0 = \alpha \cap R^\x$ for some $\alpha
    \in \Sperm R$.

    \item The final step of the proof consists in checking that
    $\sign^\CM_\alpha \subseteq J$, proving the result since $h \not \in J$.
  \end{enumerate}

  We  now proceed with the proof. 
  Let $\CS \subseteq W(R) $ be a maximal subset of Pfister forms such that:
  \begin{itemize}
    \item[(a)] $\CS \cdot \CS \subseteq \CS$;
    \item[(b)] for every $n \in \N$, $2^n \x \qf{1} \in \CS$;
    \item[(c)] for every $p \in \CS$, $(t_0^2 \x p) \cdot h\not=0$.
  \end{itemize}
  Observe that $\CS$ exists since the set $\{2^n \x \qf{1} \mid n \in \N\}$
  satisfies all these conditions.   
  Define
  \[\alpha_0 := \{r \in R^\x \mid \exists p \in \CS \quad \qf{r}\cdot p = 
  p \}.\]
  Clearly $\alpha_0 \cdot \alpha_0 \subseteq \alpha_0$.
  We first prove the following four properties:
  \begin{enumerate}
    \item[(P1)] $\alpha_0 \cup -\alpha_0 = R^\x$;
    \item[(P2)] $\alpha_0 \cap -\alpha_0 = \emptyset$;
    \item[(P3)] for every $u \in \alpha_0$, $\qf{1, u} \in \CS$;
    \item[(P4)] for every $k \in \N$ and all $u_1, \ldots, u_k \in \alpha_0$,
      $\Pf{u_1, \ldots, u_k} \in \CS$.
  \end{enumerate}
  \medskip
  
   \emph{Proof of \textrm{(P1)}}: 
    For $x \in R^\x$ we define 
    $\CS_x := \CS \cup \qf{1,x}\CS$.
    Obviously $\CS \subseteq \CS_x$. In particular $\CS_x$ satisfies property
    (b). We check the non-obvious case of    property (a): If $p,q \in \CS$,
    then
    $\qf{1,x}p\qf{1,x}q = 2\qf{1,x}pq$,
    which belongs to $\CS_x$ by properties (a) and (b) of $\CS$.
    
    Let $x \in R^\x$. If $x \not \in \alpha_0$ then $\qf{1,x} \not \in \CS$
    (otherwise, since $\qf{x} \qf{1,x} = \qf{1,x}$ we would get $x \in \alpha_0$), so
    $\CS \subsetneq \CS_x$.
  
    So if we assume that $x \not \in \alpha_0$ and $-x \not \in \alpha_0$ we
    obtain $\CS \subsetneq \CS_x$ and $\CS \subsetneq \CS_{-x}$. Since both
    $\CS_x$ and $\CS_{-x}$ satisfy properties (a) and (b), by maximality of
    $\CS$ we get that (c) does not hold for either of $\CS_x$ and $\CS_{-x}$:
    There are $p, q \in \CS$ such that $t_0^2 \x \qf{1,x}ph = 0$ and $t_0^2 \x
    \qf{1,-x}qh = 0$. Therefore $t_0^2 \x \qf{1,x}pqh = 0$ and $t_0^2 \x
    \qf{1,-x}pqh = 0$. Adding both we obtain $t_0^2 \x 2pqh = 0$ in $W(A,\s)$, a
    contradiction since $2pq \in \CS$.  End of the proof of (P1).
\medskip

     \emph{Proof of \textrm{(P2)}}: Assume that there is 
     $x \in \alpha_0 \cap -\alpha_0$. Then
    there are $p,q \in \CS$ such that $\qf{x}p=p$ and $\qf{-x}q=q$. Therefore
    $\qf{x}pq=pq$ and $\qf{-x}pq=pq$. Adding both, we get $0=2pq$ and thus $2pq
    \cdot h = 0$, a contradiction since $2pq \in \CS$. End of the proof of (P2).
\medskip

    \emph{Proof of \textrm{(P3)}}: 
    Let $u \in \alpha_0$ and let $p_0 \in \CS$ be such that $\qf{u}
    \cdot p_0 = p_0$. Consider, as in the proof of (P1), 
    $\CS_u := \CS \cup \qf{1,u}\CS$.
    As seen in the proof of (P1), $\CS_u$ satisfies properties (a) and (b). We
    check property (c): Assume $t_0^2 \x \qf{1, u}ph = 0$ for some $p \in \CS$.
    Then $t_0^2 \x \qf{1, u}p_0ph = 0$, i.e., $t_0^2 \x 2p_0ph = 0$, which is
    impossible since $2p_0p \in \CS$. Since $\CS_u$ satisfies properties (a),
    (b) and (c), and contains $\CS$, we must have $\CS = \CS_u$ by maximality of
    $\CS$. Therefore $\qf{1,u} \in \CS$. End of the proof of (P3).
\medskip

(P4) is a direct consequence of (P3) and property (a). This finishes the proof
of (P1)--(P4). 
\medskip

    Next, we define 
    \[J_\CS := \{\psi \in W(A,\s) \mid \exists p \in \CS \quad (t_0^2 \x p) \cdot
    \psi=0  \}.\]
    Clearly, $h \not \in J_\CS$ by property (c).

  Consider the map
  \[\chi \colon  R^\x \rightarrow \{-1,1\}, \quad 
    \chi(r) := \begin{cases}
      1 & \text{if } r \in \alpha_0\\
      -1 & \text{if } r \in -\alpha_0
  \end{cases}.\]
  Then $\chi$ is a signature of $R$. Indeed, it 
  is a character on $R^\x$ since $\alpha_0 \cdot \alpha_0\subseteq \alpha_0$ 
  and by properties (P1) and (P2). Furthermore, by \cite[top of p.~88]{knebusch84}, 
  $\chi$ will then be a signature of $R$ if it satisfies $\chi(-1) = -1$ 
  (which is true),
  as well as $\chi(r_1^2 f_1 + \cdots + r_k^2f_k) = 1$ whenever $f_1, \ldots, f_k
  \in R^\x$ with $\chi(f_1) = \cdots = \chi(f_k) = 1$ and $r_1, \ldots, r_k \in R$ are
  such that $z:=r_1^2 f_1 + \cdots + r_k^2f_k \in R^\x$. We check this: Since $f_1,
  \ldots, f_k \in \alpha_0$, we have by (P4) that the Pfister form $\Pf{f_1,
  \ldots, f_k}$ is in $\CS$. Clearly, $z \in D_R \Pf{f_1, \ldots, f_k}$ and by
  \cite[p.~94, Theorem~2.1]{baeza} (which applies since the hypothesis $2 \in
  R^\x$ ensures that every quotient of $R$ by a proper ideal has at least 3
  elements) we have $\qf{z} \Pf{f_1, \ldots, f_k} = \Pf{f_1, \ldots, f_k}$. Thus $z
  \in \alpha_0$ and $\chi(z)=1$.

  By the correspondence between $\Sign R$
  and $\Sperm R$ (cf. Section~\ref{sec:ord-sig})
  there is $\alpha \in \Sperm R$ such that $\chi =
  \sign_\alpha$, and thus $\alpha_0 = \alpha \cap R^\x$.
  \medskip

  \emph{Claim 1}: If $\s$ is of unitary type, then $\alpha \not \in \Nil[A,\s]$.

  Proof of Claim 1: Observe first that for every $q \in \CS$, $q \ox_R
  \qf{1}_\iota$ is not hyperbolic (if it were, then $(q \ox_R \qf{1}_\iota)
  \ox_S h$ would be hyperbolic, and thus, by Lemma~\ref{R-Z-isom}, $(q \ox_R
  \qf{1}_\iota) \ox_S h \simeq q \ox_R h$ would also be hyperbolic, contradicting the
  definition of $\CS$).

  Assume that $\alpha \in \Nil[A,\s] = \Nil[S,\iota] = \mathring H(d)$ for some
  $d \in R^\x$ (by
  Lemma~\ref{Nil-product} and Corollary~\ref{nil-open-second-kind}). Then $d \in
  \alpha \cap R^\x = \alpha_0$ and $\sign^\CM \qf{1,d} \ox_R \qf{1}_\iota = 0$
  on $\Sper R$. 
  (Indeed, the signature is zero on $\Nil[A,\s]$ by definition, and 
  the signature of $\qf{1,d}$ is zero on $\Sper R \setminus \Nil[A,\s]$ by 
  definition of $d$ and since $d\not\in \Supp \alpha$.)
  Note that $\qf{1,d} \in \CS$ by (P3).  
  Since $R$ is connected, $S$ is quadratic \'etale over $R$ and
  $\iota$ is the standard involution on $S$ by Proposition~\ref{rem:quad-et},
  and we may use Lemma~\ref{trace-sign-zero}.
  Therefore, $\sign \Tr_{S/R} (\qf{1,d} \ox_R \qf{1}_\iota) =
  0$ on $\Sper R$. Since $\Tr_{S/R} (\qf{1,d} \ox_R \qf{1}_\iota)$ is
  nonsingular (cf. Section~\ref{sec:quadalg}),  Pfister's local-global 
  principle for semilocal 
  rings (cf. \cite[p.~194]{mahe82} or \cite[Theorem~7.16]{baeza}) applies and
  there is $k \in \N$ such that $2^k \times \Tr_{S/R}(\qf{1,d} \ox_R
  \qf{1}_\iota)$ is hyperbolic. Applying \cite[Corollary~8.3]{first23} 
  yields that $2^k \times \qf{1,d} \ox_R \qf{1}_\iota$ is hyperbolic, a
  contradiction (as observed above) since $2^k \x \qf{1,d}$ is in $\CS$.  End of
  the proof of Claim 1.
  \medskip

  \emph{Claim 2}: $\ker \sign^\CM_\alpha\subseteq J_\CS$.

  Proof of Claim 2:  Let $\psi \in \ker \sign^\CM_\alpha$.
  We consider two cases.

  \begin{enumerate}

  \item $\alpha \in \Nil[A,\s]$. By Claim 1 we know that $\s$ must be of
  orthogonal or symplectic type. In particular $S=R$ and $t_0=t$.  By
  Corollary~\ref{prod-hyperbolic-no-pc2}(1), there is a nonsingular diagonal
  quadratic form $\qf{u_1, \ldots, u_{t_0}}$ of dimension $t_0$ with
  coefficients in $\alpha$ (and thus in $\alpha \cap R^\x = \alpha_0$) such that
  $\qf{u_1, \ldots, u_{t_0}} \psi$ is hyperbolic. Multiplying by $\Pf{u_1,
  \ldots, u_{t_0}}$ and using that $\qf{u_i} \Pf{u_1, \ldots, u_{t_0}} =
  \Pf{u_1, \ldots, u_{t_0}}$ (cf.
  \cite[Corollary~2.16]{baeza} which uses the fact that the hypothesis $2 \in
  R^\x$ ensures that $|R/\mathfrak{m}| > 2$ 
  for every maximal ideal $\mathfrak{m}$ of $R$), we obtain 
  $t_0 \x \Pf{u_1, \ldots, u_{t_0}} \psi = 0$.
  By property (P4) the form $\Pf{u_1, \ldots, u_{t_0}}$ is in $\CS$ and thus
  $\psi \in {J_\CS}$.

  \item
  $\alpha \not \in \Nil[A,\s]$. By Theorem~\ref{Sylvester-no-pc}, 
  $t=\rk_S A$ exists and
  we have
  \[\qf{w_1, \ldots, w_t} \psi \simeq \qf{u_1, \ldots, u_r}\qf{c}_\s \perp
  \qf{-v_1, \ldots, -v_s}\qf{c}_\s,\]
  for some $w_1, \ldots, w_t, u_1, \ldots, u_r, v_1, \ldots, v_s \in \alpha \cap
  R^\x$ and $c \in \Sym(A^\x,\s)$ such that $\sign^\CM_\alpha \qf{c}_\s =
  m_\alpha(A,\s)$. Multiplying both sides by $\Pf{w_1, \ldots, w_t}$ and then by
  $2$, the left-hand side becomes first $t \x \Pf{w_1, \ldots, w_t} \psi$ (using
  that $\qf{w_i}\Pf{w_1, \ldots, w_t} = \Pf{w_1, \ldots, w_t}$) and then $t_0 \x
  \Pf{w_1, \ldots, w_t} \psi$, while the right-hand side still retains the same
  shape (up to taking larger values for $r$ and $s$, and different elements
  $u_i, v_j \in \alpha \cap R^\x$). We thus have
  \begin{equation*}
    t_0 \x \Pf{w_1, \ldots, w_t} \psi \simeq \qf{u_1, \ldots, u_{r}}\qf{c}_\s \perp
    \qf{-v_1, \ldots, -v_{s}}\qf{c}_\s,
  \end{equation*}
  where $u_1, \ldots, u_{r}, v_1, \ldots, v_{s} \in \alpha \cap
  R^\x$.

  Since $\sign^\CM_\alpha \psi = 0$ and $\sign^\CM_\alpha \qf{u_i} \qf{c}_\s =
  \sign^\CM_\alpha \qf{v_j}\qf{c}_\s = m_\alpha(A,\s)$ for $i=1,\ldots,r$ and
  $j=1,\ldots, s$, we must have $r=s$. In particular we
  can pair each $u_ic$ with
  the corresponding $-v_ic$, so that
  \begin{equation}\label{step}
    t_0 \x \Pf{w_1, \ldots, w_t} \psi \simeq \bigperp_{i=1}^r \qf{u_ic,
    -v_ic}_\s.
  \end{equation}

  \emph{Fact}: For each $i=1, \ldots, r$ there is $p_i \in \CS$ such that 
  $t_0 \x p_i \cdot \qf{u_ic,-v_ic}_\s =0$.\\
  Proof of the Fact: By Corollary~\ref{prod-hyperbolic-no-pc2}(2),  
  there are $z_1, \ldots, z_t$, $r_1, \ldots, r_\ell \in \alpha \cap
  R^\x = \alpha_0$ such that $\qf{z_1, \ldots, z_t} \Pf{r_1, \ldots, r_\ell}
  \qf{u_ic,-v_ic}_\s = 0$. Multiplying by the Pfister form $\Pf{z_1, \ldots,
  z_t}$, we obtain
  \[(t \x \Pf{z_1, \ldots, z_t, r_1, \ldots, r_\ell}) \qf{u_ic, -v_ic}_\s = 0,\]
  and thus
  $(t_0 \x \Pf{z_1, \ldots, z_t, r_1, \ldots, r_\ell}) \qf{u_ic, -v_ic}_\s = 0$.
  The Fact follows since the form $\Pf{z_1, \ldots, z_t, r_1, \ldots, r_\ell}$ is in
  $\CS$ by property (P4). End of the proof of the Fact.
  \medskip

  Multiplying \eqref{step} by $t_0 \x p_1 \cdots p_r$ gives
  $t_0^2 \x (p_1 \cdots p_r \Pf{w_1, \ldots, w_t}) \cdot \psi = 0$,
  proving that $\psi$ is in $J_\CS$ since $p_1 \cdots p_r \Pf{w_1, \ldots, w_t}
  \in \CS$ (by properties (P4) and (a)). End of the proof of Claim~2.
  \end{enumerate}

  The conclusion is now clear since $h \not \in J_\CS$.
\end{proof}

\begin{prop}\label{prop:tors}
  Let $R$ be semilocal connected.
  The torsion in $W(A,\s)$ is $2$-primary.
\end{prop}

\begin{proof}
  By \cite[Theorem~8.7 and Remark~8.8]{BFP} there exists a connected finite
  \'etale $R$-algebra $R_1$ of odd rank, and an Azumaya algebra with involution
  $(A_1,\s_1)$ over $R_1$ such that $A\ox_R R_1$ and $A_1$ are Brauer
  equivalent over $S_1:=S\ox_R R_1$, 
  $\s$ and $\s_1$ are both unitary, or are both non-unitary, 
  and such
  that at least one of the following holds:
  \begin{enumerate}
    \item $Z(A_1)\cong R_1\x R_1 $;
    \item $\deg A_1=1$;
    \item The index and degree of $A_1$ are equal and divide the index of $A$. 
    Moreover, $\deg A_1$ is a power of $2$ and there exist 
    $u,v \in A_1^\x$ such that $u^2 \in R_1^\x$, $\s_1(u)=-u$, $\s_1(v)=-v$
    and $uv=-vu$.
  \end{enumerate}

  By \cite[Corollary~7.4]{BFP}, the canonical map of Witt groups
  \[
    W(A,\s) \to W(A\ox_R R_1, \s\ox \id_{R_1})
  \] 
  is injective and thus it suffices to show that the torsion in
  $W(A\ox_R R_1, \s\ox \id_{R_1})$ is $2$-primary. 
  
  If $S_1$ is not connected, then $W(A\ox_R R_1, \s\ox \id_{R_1})=0$ (cf. 
  Remark~\ref{rem:notcon}) and we can conclude. 
  Thus we may assume that $S_1$ is connected, and in particular that we are
  not in case (1) above.
  By Theorem~\ref{thm:HME}
  we have a Witt group isomorphism
  \[
    W(A\ox_R R_1, \s\ox \id_{R_1}) \cong W^\delta (A_1, \s_1)
  \]
  for some $\delta \in\{-1,1\}$, where we may take $\delta=1$ if $\s$ 
  and $\s_1$ are unitary, observing that 
  $\s$ and $\s\ox \id_{R_1}$ are of the same type. 
  We now examine the remaining relevant cases from \cite[Theorem~8.7]{BFP},
  as listed in (2) and (3) above:
  
  \begin{enumerate}
    \addtocounter{enumi}{1}
    \item $\deg A_1=1$, i.e., $A_1=S_1$: Assume first that $\s_1$ is not 
    unitary. By Proposition~\ref{rem:quad-et} we then have $A_1=S_1=R_1$
    and $\s_1=\id_{S_1}$.  In this case
    $W^{-1}(R_1,\id)=0$ by \cite[I, Corollary~4.1.2]{knus91} (whose hypotheses
    are satisfied
    since $R_1$ is connected and also semilocal by 
    \cite[VI, Proposition~1.1.1]{knus91}), 
    while the torsion in $W(R_1,\id)$
    is $2$-primary by \cite[Chapter V, Theorem~6.6]{baeza}. 
    
    On the other hand, if $\s_1$ is unitary (so that we may take $\delta=1$),
    then $S_1$ is a quadratic \'etale 
    $R_1$-algebra and $\s_1$ is the standard involution. 
    By \cite[Corollary~8.3]{first23},
    the map $h \mapsto \Tr_{S_1/R_1} \circ h$ is an injection from 
    $W(S_1, \s_1)$ into 
    $W(R_1, \id)$, which has $2$-primary torsion as observed above.

    \item $\deg A_1$ is a power of $2$ and there exists $u\in A_1^\x$ such that 
    $\s_1(u)=-u$. In particular, $\rk_{S_1} A_1$ and hence $\rk_{R_1} A_1$ 
    are powers of $2$.  
    We consider two cases:
    
    If $\delta = 1$,  we conclude with 
    Proposition~\ref{PLG-V1-no-pc}:
    If $h$ is torsion in $W(A_1,\s_1)$, then $h$ has zero signature at every
    ordering of $R_1$ and by Proposition~\ref{PLG-V1-no-pc},
    there is $n \in \N$ such that $2^n \x h$ is hyperbolic.

    If $\delta = -1$, we have
      \[
          W^{-1}(A_1,\s_1)\cong W(A_1, \Int(u)\circ \s_1)
        \]
    (by Morita equivalence, more precisely 
    the $\mu$-conjugation equivalence of categories in
    \cite[Section~2.7]{first23}), and we conclude again with 
    Proposition~\ref{PLG-V1-no-pc} applied to
    $(A_1, \Int(u)\circ \s_1)$.\qedhere  
  \end{enumerate}
\end{proof}

\begin{rem}\label{sl-prod-no-pc}
  If $R$ is semilocal with only $k$ maximal ideals, then any expression of $R$
  as a product $R_1 \x \cdots \x R_n$ of rings must be such that $n \le k$. Therefore there is
  such an expression of $R$ as a product $R_1 \x \cdots \x R_t$ that cannot be
  further decomposed as a product and thus where each $R_i$ is connected.
\end{rem}

\begin{thm}[Pfister's local-global principle]\label{PLG-Az}
  Let $R$ be semilocal, and recall that we assume $2 \in R^\x$. Let $(M,h)$ be a nonsingular hermitian
  form over $(A,\s)$. The following statements are equivalent:
  \begin{enumerate}
    \item $\sign^\CM_\alpha h = 0$ for every $\alpha \in \Sper R$.
    \item $\sign^\CM_\alpha h = 0$ for every $\alpha \in \Sperm R$.
    \item There exists $n \in \N\cup\{0\}$ such that $2^n \x h$ is hyperbolic.
  \end{enumerate}
  In particular, the torsion in $W(A,\s)$ is $2$-primary.
\end{thm}
\begin{proof}
  Observe that the final statement clearly follows from the equivalence of (1) 
  and (3), since a torsion form has zero signature at every ordering.

  Clearly (1) implies (2), and (3) implies (1), 
  so we only need to show that (2)
  implies (3).

  Following Remark~\ref{sl-prod-no-pc}, we may assume that $R = R_1 \x \cdots \x R_t$
  with $R_1, \ldots, R_t$ connected semilocal rings.
  Writing $e_1=(1,0,\ldots,0), \ldots, e_t=(0,\ldots,0,1)$ in $R$, we have
  \[
    (A,\s) \cong (Ae_1, \s|_{Ae_1})\x \cdots \x (Ae_t, \s|_{Ae_t}).
  \]
  Furthermore, we
  can
  identify $M$ with $\bigoplus_{i=1}^t M e_i$, and we consider
  $h_i := h|_{Me_i}$ as a hermitian form over $(Ae_i, \s|_{Ae_i})$ for
  $i=1,\ldots, t$.
  A direct verification shows that:
  \begin{itemize}
    \item $h$ is nonsingular if and only if each $h_i$ is nonsingular for $i=1,
      \ldots, t$ (using for instance that $h$ is nonsingular if and only if for
      every maximal ideal $\mathfrak{m}$ of $R$
      the form $h \ox_R R/{\mathfrak{m}}$ is nonsingular, cf.
      \cite[I, Lemma~7.1.3]{knus91}).
    \item  If  $h_i$ is hyperbolic for $i=1, \ldots, t$, then $h$ is hyperbolic.
    Indeed: We can write
    $Me_i = L_i \oplus P_i$ with $h_i(L_i,L_i) = 0$ and $h_i(P_i,P_i) =
    0$ (cf. \cite[Section~2.2]{first23}).
    So $M = \bigoplus_{i=1}^t L_i \oplus P_i \cong (\bigoplus_{i=1}^t L_i)
    \oplus (\bigoplus_{i=1}^t P_i)$. Let $L = \bigoplus_{i=1}^t L_i$ and $P =
    \bigoplus_{i=1}^t P_i$. We check that $h(L,L) = 0$; the proof of
    $h(P,P) = 0$
    is similar. It suffices to show that $h(\ell_i,\ell'_j) = 0$ for each
    $\ell_i \in L_i$
    and $\ell'_j \in L_j$. If $i \not = j$ then $h(\ell_i,\ell'_j) = 0$ and if
    $i = j$ then
    $h(\ell_i,\ell'_i) = h_i(\ell_i,\ell'_i) = 0$.
  \end{itemize}

  Every $\alpha \in \Sperm R_i$ can be seen as an element $\alpha'$ in $\Sperm
  R$, and a direct verification of the definition of signature shows that
  $\sign^\CM_{\alpha'} h = \sign^\CM_\alpha h_i$.  Therefore, (2) gives that
  $\sign^\CM_\alpha h_i = 0$ for every $\alpha \in \Sperm R_i$ and every $i=1,
  \ldots, t$. By Propositions~\ref{PLG-V1-no-pc} and \ref{prop:tors}
  we have that for every
  $i=1,\ldots, t$ there exists $n_i \in \N$ such that $2^{n_i} \x h_i$ is
  hyperbolic. Thus, letting $n=n_1+\cdots + n_t$, it follows that $2^n\x h_i$ is
  hyperbolic for $i=1,\ldots,t$ and hence that $2^n\x h$ is 
  hyperbolic. 
\end{proof}

\section*{Acknowledgements}

We thank Igor Klep for many stimulating exchanges of ideas, and
are very grateful to the Fondazione Bruno Kessler--Centro Internazionale per la
Ricerca Matematica and Augusto Micheletti for hosting the three of us in
December 2018 for a Research in Pairs stay during which
some of the ideas that eventually led to this paper were discussed.

We are very grateful to the Referees for reading our paper in great detail,
for spotting a serious mistake,
and for the many pertinent and helpful questions and suggestions, which led
to substantial improvements.

\end{document}